\documentclass[11pt]{article}
\usepackage{amsfonts,amssymb,color}
\usepackage[mathscr]{eucal}
\usepackage{amsmath, amsthm}
\usepackage{mathrsfs}
\usepackage{amsbsy}
\usepackage{dsfont}
\usepackage{bbm}
\usepackage{wasysym}
\usepackage{stmaryrd}
\usepackage{upgreek}
\usepackage{hyperref}
\usepackage{comment}
\input xypic
\xyoption{all}
\begin{document}
\baselineskip = 5mm
\newcommand \primes {{\bf P}}
\newcommand \ZZ {{\mathbb Z}}
\newcommand \FF {{\mathbb F}}
\newcommand \NN {{\mathbb N}}
\newcommand \QQ {{\mathbb Q}}
\newcommand \RR {{\mathbb R}}
\newcommand \CC {{\mathbb C}}
\newcommand \PR {{\mathbb P}}
\newcommand \AF {{\mathbb A}}
\newcommand \VV {{\mathbb V}}
\newcommand \ud {{\Bbbk }}
\newcommand \bgf {{\mathbb K}}
\newcommand \bcA {{\mathscr A}}
\newcommand \bcB {{\mathscr B}}
\newcommand \bcC {{\mathscr C}}
\newcommand \bcD {{\mathscr D}}
\newcommand \bcE {{\mathscr E}}
\newcommand \bcF {{\mathscr F}}
\newcommand \bcG {{\mathscr G}}
\newcommand \bcH {{\mathscr H}}
\newcommand \bcM {{\mathscr M}}
\newcommand \bcN {{\mathscr N}}
\newcommand \bcI {{\mathscr I}}
\newcommand \bcJ {{\mathscr J}}
\newcommand \bcK {{\mathscr K}}
\newcommand \bcL {{\mathscr L}}
\newcommand \bcO {{\mathscr O}}
\newcommand \bcP {{\mathscr P}}
\newcommand \bcQ {{\mathscr Q}}
\newcommand \bcR {{\mathscr R}}
\newcommand \bcS {{\mathscr S}}
\newcommand \bcT {{\mathscr T}}
\newcommand \bcU {{\mathscr U}}
\newcommand \bcV {{\mathscr V}}
\newcommand \bcW {{\mathscr W}}
\newcommand \bcX {{\mathscr X}}
\newcommand \bcY {{\mathscr Y}}
\newcommand \bcZ {{\mathscr Z}}
\newcommand \goa {{\mathfrak a}}
\newcommand \gob {{\mathfrak b}}
\newcommand \goc {{\mathfrak c}}
\newcommand \god {{\mathfrak d}}
\newcommand \gom {{\mathfrak m}}
\newcommand \gon {{\mathfrak n}}
\newcommand \goo {{\mathfrak o}}
\newcommand \gop {{\mathfrak p}}
\newcommand \goq {{\mathfrak q}}
\newcommand \goA {{\mathfrak A}}
\newcommand \goB {{\mathfrak B}}
\newcommand \goT {{\mathfrak T}}
\newcommand \goC {{\mathfrak C}}
\newcommand \goD {{\mathfrak D}}
\newcommand \goM {{\mathfrak M}}
\newcommand \goN {{\mathfrak N}}
\newcommand \goO {{\mathfrak O}}
\newcommand \goP {{\mathfrak P}}
\newcommand \goQ {{\mathfrak Q}}
\newcommand \goS {{\mathfrak S}}
\newcommand \goH {{\mathfrak H}}
\newcommand \ga {{\mathfrak a}}
\newcommand \gb {{\mathfrak b}}
\newcommand \gc {{\mathfrak c}}
\newcommand \gd {{\mathfrak d}}
\newcommand \gm {{\mathfrak m}}
\newcommand \gn {{\mathfrak n}}
\newcommand \gp {{\mathfrak p}}
\newcommand \gq {{\mathfrak q}}
\newcommand \gQ {{\mathfrak Q}}
\newcommand \gP {{\mathfrak P}}
\newcommand \gT {{\mathfrak T}}
\newcommand \gC {{\mathfrak C}}
\newcommand \gD {{\mathfrak D}}
\newcommand \gM {{\mathfrak M}}
\newcommand \gS {{\mathfrak S}}
\newcommand \gH {{\mathfrak H}}
\newcommand \gA {{\rm {A}}}
\newcommand \gB {{\rm {B}}}
\newcommand \catC {{\sf C}}
\newcommand \catD {{\sf D}}
\newcommand \catF {{\sf F}}
\newcommand \catG {{\sf G}}
\newcommand \catE {{\sf E}}
\newcommand \catI {{\sf I}}
\newcommand \catS {{\sf S}}
\newcommand \catW {{\sf W}}
\newcommand \catX {{\sf X}}
\newcommand \catY {{\sf Y}}
\newcommand \catZ {{\sf Z}}
\newcommand \Sets {{\sf Sets}}
\newcommand \Sch {{\sf Sch }}
\newcommand \Funct {{\rm Funct}}
\newcommand \PShv {{\rm PSh}}
\newcommand \Shv {{\rm Shv }}
\newcommand \APSh {{\sf APSh}}
\newcommand \AShv {{\sf AShv }}
\newcommand \Ringspace {{\sf {Ringspace}}}
\newcommand \Nor {{\sf Nor}}
\newcommand \Seminor {{\sf sNor}}
\newcommand \Reg {{\sf Reg}}
\newcommand \Sm {{\sf Sm}}
\newcommand \SmProj {{\sf SmProj}}
\newcommand \NorCon {{\sf NorCon}}
\newcommand \Noe {{\sf Noe}}
\newcommand \LocNoe {{\sf LocNoe}}
\newcommand \Stk {{\sf Stk}}
\newcommand \Mon {{\rm Mon }}
\newcommand \Mod {{\sf Mod}}
\newcommand \Ab {{\rm Ab }}
\newcommand \Ind {{\sf {Ind}}}
\newcommand \Vect {{\sf Vect}}
\newcommand \MM {{\sf MM}}
\newcommand \HS {{\sf HS}}
\newcommand \MHS {{\sf MHS}}
\newcommand \Zar {\rm {Zar}}
\newcommand \Nis {\rm {Nis}}
\newcommand \Nen {\rm {N\acute en}}
\newcommand \cdh {\rm {cdh}}
\newcommand \h {\rm {h}}
\newcommand \fppf {{\rm {fppf}}}
\newcommand \et {\rm {\acute e t}}
\newcommand \CHM {{\sf Chow}}
\newcommand \DM {{\sf DM}}
\newcommand \Ta {{\mathbbm T}}
\newcommand \uno {{\mathbbm 1}}
\newcommand \Le {{\mathbbm L}}
\newcommand \ptr {{\pi _2^{\rm tr}}}
\newcommand \Spec {{\rm {Spec}}}
\newcommand \bSpec {{\bf {Spec}}}
\newcommand \Pic {{\rm {Pic}}}
\newcommand \PicF {{\it {Pic}}}
\newcommand \PicS {{\rm {Pic}}}
\newcommand \Jac {{\rm {Jac}}}
\newcommand \Alb {{\rm {Alb}}}
\newcommand \alb {{\rm {alb}}}
\newcommand \NS {{{\rm NS}}}
\newcommand \Corr {{Corr}}
\newcommand \Sym {{\rm {Sym}}} 
\newcommand \Symcl {{\rm {S}}} 
\newcommand \Alt {{\rm {Alt}}}
\newcommand \Prym {{\rm {Prym}}}
\newcommand \HilbF {{\rm Hilb}}
\newcommand \HilbS {{\rm Hilb}}
\newcommand \Mor {{\rm Mor}}
\newcommand \MorF {{\rm Mor}}
\newcommand \MorS {{\rm Mor}}
\newcommand \pdiv {{\rm {div}}}
\newcommand \Bl {{\rm {Bl}}}
\newcommand \codim {{\rm {codim}}}
\newcommand \Proj {{\rm {Proj}}}
\newcommand \bProj {{\bf {Proj}}}
\newcommand \Div {{\rm {Div}}}
\newcommand \prim {{\rm {prim}}}
\newcommand \stalk {{\rm st}}
\newcommand \seminorm {{\rm {sn}}}
\newcommand \hc {{\rm hc}} 
\newcommand \BC {{\rm {BC}}}
\newcommand \cnj {{\rm {C}}}
\newcommand \Res {{\rm {Res}}}
\newcommand \tor {{\rm {tor}}}
\newcommand \Cycl {{\it Cycl }}
\newcommand \PropCycl {{\it PropCycl }}
\newcommand \cycl {{\it cycl }}
\newcommand \PrimeCycl {{\it PrimeCycl }}
\newcommand \PrimePropCycl {{\it PrimePropCycl }}
\newcommand \card {{\rm {card}}}
\newcommand \cone {{\rm {cone}}}
\newcommand \cha {{\rm {char}}}
\newcommand \eff {{\rm {eff}}}
\newcommand \cl {{\rm {cl}}}
\newcommand \tr {{\rm {tr}}}
\newcommand \pr {{\rm {pr}}}
\newcommand \ev {{e}}
\newcommand \interior {{\rm {Int}}}
\newcommand \sep {{\rm {sep}}}
\newcommand \td {{\rm {tdeg}}}
\newcommand \alg {{\rm {alg}}}
\newcommand \im {{\rm im}}
\newcommand \rd {{\rm {red}}}
\newcommand \hl {{\rm h}}
\newcommand \shl {{\rm sh}}
\newcommand \op {{\rm op}}
\newcommand \Hom {{\rm Hom}}
\newcommand \uHom {{\underline {\rm Hom}}}
\newcommand \cHom {{\mathscr H\! }{\it om}}
\newcommand \Ext {{\rm Ext}}
\newcommand \cExt {{\mathscr E\! }{\it xt}}
\newcommand \colim {{{\rm colim}\, }} 
\newcommand \End {{\rm {End}}}
\newcommand \coker {{\rm {coker}}}
\newcommand \id {{\rm {id}}}
\newcommand \van {{\rm {van}}}
\newcommand \spc {{\rm {sp}}}
\newcommand \Ob {{\rm Ob}}
\newcommand \Aut {{\rm Aut}}
\newcommand \cor {{\rm {cor}}}
\newcommand \res {{\rm {res}}}
\newcommand \Gal {{\rm {Gal}}}
\newcommand \PGL {{\rm {PGL}}}
\newcommand \Gr {{\rm {Gr}}}
\newcommand \Tor {{\rm {Tor}}}
\newcommand \Sing {{\rm {Sing}}}
\newcommand \spn {{\rm {span}}}
\newcommand \univ {{\rm {\, univ}}}
\newcommand \Nm {{\rm {Nm}}}
\newcommand \fin {{\rm {f}}}
\newcommand \inv {{\rm {inv}}}
\newcommand \even {{\rm {even}}}
\newcommand \md {{\rm {mod}\, }}
\newcommand \sg {{\Sigma }}
\newcommand \ind {{\rm {ind}}}
\newcommand \Gm {{{\mathbb G}_{\rm m}}}
\newcommand \trdeg {{\rm {tr.deg}}}
\newcommand \con {\rm {conn}}
\newcommand \sv {{\rm {sv}}}
\newcommand \sing {{\rm {sing}}}
\newcommand \tame {\rm {t}}
\newcommand \eq {{\rm {eq}}}
\newcommand \length {{\rm {length}}}
\newcommand \ord {{\rm {ord}}}
\newcommand \shf {{\rm {a}}}
\newcommand \spd {{\rm {s}}}
\newcommand \glue {{\rm {g}}}
\newcommand \equi {{\rm {equi}}}
\newcommand \ab {{\rm {ab}}}
\newcommand \add {{\rm {ad}}}
\newcommand \Fix {{\rm {Fix}}}
\newcommand \pty {{\mathbf P}}
\newcommand \type {{\mathbf T}}
\newcommand \trp {{\rm {t}}}
\newcommand \cat {{\rm {cat}}}
\newcommand \deop {{\Delta \! }^{op}\, }
\newcommand \defect {{\rm {def}}}
\newcommand \aff {{\rm {aff}}}
\newcommand \Const {{\rm {Const}}}
\newcommand \num {{\rm {num}}}
\newcommand \conv {{\it {cv}}}
\newcommand \nil {{\rm {nil}}}
\newcommand \rat {{\rm rat}}
\newcommand \gen {{\rm gen}}
\newcommand \SG {{\rm SG}}
\newcommand \tors {{\rm {tors}}}
\newcommand \coeq {{{\rm coeq}\, }}
\newcommand \supp {{\rm Supp}}
\newcommand \sm {{\rm sm}}
\newcommand \reg {{\rm reg}}
\newcommand \nor {{\rm nor}}
\newcommand \noe {{\rm Noe}}
\newcommand \var {{\rm var}}
\newcommand \norm {{\rm {norm}}}
\newcommand \tp {{\rm {tp}}}
\newcommand \st {{\rm {st}}}
\newcommand \Gys {{{G}}}
\newcommand \Fr {{\rm {Fr}}}
\newcommand \hol {{\rm {h}}}
\newcommand \reduced {{\rm {red}}}
\newcommand \ttop {{\rm top}}
\newcommand \AJ {{\rm AJ}}
\newcommand \Rat {{\mathscr Rat}}
\newcommand \tdf {{\mathbf {t}}}
\newcommand \term {{*}}
\newcommand \znak {{\natural }}
\newcommand \znakk {{\sharp }}
\newcommand \znakkk {{\flat }}
\newcommand \qand {{\quad \hbox{and}\quad }}
\newcommand \qqand {{\quad \quad \hbox{and}\quad \quad }}
\newcommand \qqqand {{\quad \quad \quad \hbox{and}\quad \quad \quad }}
\newcommand \heither {{\hbox{either}\quad }}
\newcommand \qor {{\quad \hbox{or}\quad }}
\newcommand \qqor {{\quad \quad \hbox{or}\quad \quad }}
\newcommand \lra {\longrightarrow}
\newcommand \hra {\hookrightarrow}
\def\blue {\color{blue}}
\def\red {\color{red}}
\def\green {\color{green}}
\newtheorem{theorem}[subsubsection]{Theorem}
\newtheorem{lemma}[subsubsection]{Lemma}
\newtheorem{corollary}[subsubsection]{Corollary}
\newtheorem{proposition}[subsubsection]{Proposition}
\newtheorem{remark}[subsubsection]{Remark}
\newtheorem{definition}[subsubsection]{Definition}
\newtheorem{conjecture}[subsubsection]{Conjecture}
\newtheorem{example}[subsubsection]{Example}
\newtheorem{question}[subsubsection]{Question}
\newtheorem{mycomment}[subsubsection]{Comment}
\newtheorem{assumption}[subsubsection]{Assumption}
\newtheorem{fact}[subsubsection]{Fact}
\newtheorem{crucialquestion}[subsubsection]{Crucial Question}
\newtheorem{claim}[subsubsection]{Claim}
\newtheorem{aim}[subsubsection]{Aim}
\newenvironment{pf}{\par\noindent{\em Proof}.}{\hfill\framebox(6,6)
\par\medskip}

\title{\bf Arithmetic of Gysin kernels}





\author{\sc Vladimir Guletski\u \i \, {\small \it and}\,  Bo Zhang}


\maketitle

\begin{abstract}
\noindent Let $k$ be a field, and let $X$ be a smooth projective surface over $k$. Fix a Lefschetz pencil on $X$, and let $C$ be its fibre at the generic point of $\PR ^1$. The closed immersion of $C$ in to $X_{k(t)}$ induces the Gysin homomorphism from the Jacobian $A$ of the curve $C$ to the Chow group $A_0(X_{k(t)})$ of $0$-cycles of degree $0$ on $X_{k(t)}$. Embedding $k(t)$ in to an uncountable universal domain $\ud $, we obtain the corresponding homomorphism from $A(\ud )$ to $A_0(X_{\ud })$, whose kernel is either countable or the union of translates of a certain abelian subvariety inside $A_{\ud }$, due to the Deligne-Katz irreducibility of monodromy action on vanishing cycles. We prove three dichotomy theorems on the structure of the kernel of Gysin homomorphism on $0$-cycles, in terms of \'etale monodromy action, and working over a field of arbitrary characteristic.
\end{abstract}
















\tableofcontents

\section{Introduction}
\label{intro}

Let $k$ be a field, let $X$ be a smooth projective surface over $k$, and let $C$ be a smooth projective curve inside $X$, such that $C(k)\neq \emptyset $. The closed immersion $\iota $ of $C$ in to $X$ indices the push-forward Gysin homomorphism 
  $$
  \iota _*:A_0(C)\to A_0(X)
  $$ 
on Chow groups of $0$-cycles of degree $0$ modulo rational equivalence, and the source group is canonically isomorphic to the group $A(k)$ of $k$-rational points of the Jacobian        
  $$
  A=\Jac _{C/k}
  $$ 
of the curve $C$ over $k$. Therefore, we obtain the push-forward Gysin homomorphism 
  $$
  \iota _*:A(k)\to A_0(X)\; ,
  $$
whose kernel 
  $$
  \Gys _0(k)=\ker (\iota _*)
  $$
is the main object of study in this paper. 


The first observation is that if 
  $$
  \alb :A_0(X)\to \Alb _{X/k}(k)
  $$
is the Albanese homomorphism from the Chow group $A_0(X)$ to the group of $k$-rational points of the Albanese variety of the surface $X$, the Gysin kernel $\Gys _0(k)$ is contained in the bigger kernel
  $$
  \Gys _1(k)=\ker (\alb \circ \iota _*)\; .
  $$

Moreover, there exists a regular morphism from $A$ to $\Alb _{X/k}$ over $k$, such that the composition $\alb \circ \iota _*$ is the application of it to points rational over $k$. Consequently, the kernel $\Gys _1(k)$ is the set of $k$-rational points of some Zariski closed group subscheme in $A$. Let $A_1$ be the connected component of this subscheme containing $0$. Then $A_1$ is an abelian subvariety in $A$, and $A_1(k)$ is a subset in $\Gys _1(k)$. 

Assume now that $k$ is embedded in to an uncountable universal domain $\ud $, in the sense of Weil, see \cite{Weil}. If $\cha (k)=0$ then $\ud $ can be $\CC $, and if $\cha (k)=p>0$ then $\ud $ can be the algebraic closure $\overline {\FF _p((t))}$ of the field of Laurent series with coefficients in $\FF _p=\ZZ /p$. The latter may be also considered as the field of generalised Puiseux series over the finite field $\FF _p$, see \cite{Kedlaya}. 

Consider the closed immersion $\iota _{\ud }$ of $C_{\ud }$ in to $X_{\ud }$. Then  there exists an abelian subvariety $A_0$ inside $(A_1)_{\ud }$ over $\ud $, such that the kernel $\Gys _0(\ud )$ is the union of a countable collection of translates of the group $A_0(\ud )$ inside $A(\ud )$. 

Assume $k$ is algebraically closed, and choose and fix a Lefschetz pencil on $X$ over $k$, see Prop. 1.5 on page 177 in \cite{FreitagKiehl} or Expos\'e XVII in \cite{SGA7-2}. Resolving the indeterminacy locus, we may also assume, without loss of generality, that our pencil is a regular morphism 
  $$
  X\to \PR ^1
  $$ 
having a section over $k$, see Theorem 31.3 on page 185 in \cite{MilneLEC} or Expos\'e XVII in \cite{SGA7-2}.

Extending scalars from $k$ to $\ud $, we also obtain the pencil over $\ud $. Let $\eta $ be the generic point of the projective line, and let $\bar \eta $ be the geometric generic point on it. Let $C$ be the generic fibre of the pencil, and let $C_{\bar \eta }$ be the geometric generic fibre of the pencil. The closed immersion of $C_{\bar \eta }$ in to $X_{\bar \eta }$ gives rise to the Gysin kernel $\Gys _0(\bar \eta )$, and the abelian varieties $A_{\bar \eta ,1}$ and $A_{\bar \eta ,0}$.

\medskip

Our first result is this:

\medskip

\begin{itemize}
\item[]{}
{\rm THEOREM A.} {\it
In terms above, either 
  $$
  A_{\bar \eta ,0}=0
  $$ 
or 
  $$
  A_{\bar \eta ,0}=A_{\bar \eta ,1}\; .
  $$
}
\end{itemize}

\medskip

Notice also that, since $\ud $ is isomorphic, not canonically, to $\overline {\ud (t)}$, the same dichotomy is true for any transcendental number as a point on $\PR ^1_{\ud }$.

Theorem A is similar to what we proved in \cite{Kalyan&Vovan} in characteristic $0$. On the other hand, the present theorem is proven over the ground field of arbitrary characteristic, and our arguments are now simpler than in \cite{Kalyan&Vovan}. It should be also pointed out that the same result is proven in \cite{SchoemannWerner}. The reader is cordially invited to compare the methods.


Over $\CC $, Theorem A plays an important role in the study of rational equivalence of $0$-cycles, when $C$ in moving in an ample linear system on $X$, see pages 304 - 305 in the second volume of the book \cite{VoisinBook}, and Proposition 2.4 on page 854 in \cite{VoisinSymplInv}. 

Recall that the Chow group $A_0(X_{\ud })$ is said to be (weakly) representable, if it can be covered by an abelian variety, in some regular sense, see Definition 1.1 in \cite{BlochAnExample} or Definition 3.3 in \cite{BlochMurre}. If the second cohomology group $H^2(X,\QQ _l)$ is algebraic, i.e. it can be generated by cycles classes of divisors on $X$, then weak representability of $A_0(X)$ is equivalent to Bloch's conjecture, for the surface $X$, see \cite{BCinitial} and \cite{BlochLectures}. 

The next theorem tells us that weak representability of $A_0(X_{\ud })$ is equivalent to the property of $\Gys _0$ being countable at the geometric generic point $\bar \eta $, provided $H^1(X)=H^3(X)=0$. 


\medskip

\begin{itemize}
\item[]{}
{\rm THEOREM B.} {\it
In terms above, if $A_0(X_{\ud })$ is not weakly representable, then the kernel $\Gys _0(\overline {\ud (t)})$ is countable. Assuming the Gysin homomorphism
  $$
  H^1(C_{\bar \eta },\QQ _l)\to H^3(X_{\bar \eta },\QQ _l(1))
  $$
is not injective, the converse is also true: if $\Gys _0(\overline {\ud (t)})$ is countable, then $A_0(X_{\ud })$ is not weakly representable. 
}
\end{itemize}

\medskip

Finally, our third result is the following 

\medskip

\begin{itemize}
\item[]{}
{\rm THEOREM C.} {\it
Let $k_0$ be the minimal field of definition of the Lefschetz pencil, and let $k_0(t)$ is the purely transcendental field extension. If the kernel $\Gys _0(\overline {\ud (t)})$ is countable, then 
  $$
  \Gys _0(\overline {\ud (t)})=\Gys _0(\overline {k_0(t)})\; ,
  $$
i.e. all points in the Gysin kernel are algebraic over $k_0(t)$.
}
\end{itemize}

\medskip

The meaning of Theorem C is that in order to prove Bloch's conjecture, provided $H^2(X,\QQ _l)$ is algebraic, it suffices to find a point of transcendence degree $>0$ in the Gysin kernel $\Gys _0(\overline {\ud (t)})$. 

Another application would be to $K3$-surfaces over $\QQ $ or, say, $\FF _p$. Since $A_0(X_{\ud })$ is known to be not representable, see \cite{Mumford} and \cite{BlochLectures}, all points in the Gysin kernel are rational over $\overline {\QQ (t)}$ or, respectively, over $\overline {\FF_p (t)}$.



\bigskip

\section{The Mumford-Ro\u \i tman lemma}
\label{MRT}
 
 In \cite{Mumford}, Mumford proved a certain lemma saying that rational equivalence of $0$-cycles on a variety $X$ is encoded by a countable collection of closed subschemes in the symmetric powers of $X$, when the ground field is $\CC $. Later on, Ro\u \i tman gave a detailed proof of Mumford's lemma in \cite{RoitmanGamma}, also working over $\CC $. Our first aim in the present paper is to convince the reader that the Mumford-Ro\u \i tman theorem holds true in positive characteristic too.

\subsection{The Suslin-Voevodsky representability theorem}

First we need to recall representability of $0$-cycles by symmetric powers proven in rigour, in both zero and positive characteristics, by Suslin and Voevodsky in \cite{SuslinVoevodsky}. 

Let $k$ be a field, and let $p$ be the exponential characteristic of $k$. Let $X$ be a scheme\footnote{all schemes in this paper are separated by default} of finite type over $k$. Let $\ZZ [1/p]$ be the minimal subring containing integers and $1/p$ in $\QQ $, and for any normal integral scheme $S$ over $k$, let 
  $$
  z_0(X)(S)
  $$ 
be the free $\ZZ [1/p]$-module generated by closed integral subschemes 
  $$
  Z\subset X\times S
  $$ 
whose projection to $S$ is finite and surjective. In other words, $z_0(X)(S)$ is the module of relative $0$-cycles on $X\times S$ over $S$ with coefficients in $\ZZ [1/p]$. 

Notice that the inversion of $p$ is completely necessary, as it allows us to correctly define pullbacks in case of characteristic $p>1$. For any morphism 
  $$
  f:S'\to S
  $$ 
 there exists a well defined pullback homomorphism 
   $$
   f^*:z_0(X)(S)\to z_0(X)(S')\; ,
   $$
and thus we obtain the pre-sheaf $z_0(X)$ of $0$-cycles on the category $\Nor $ of integral normal schemes over $k$, see page 79 in \cite{SuslinVoevodsky}. 

Let also 
  $$
  z_0^{\eff }(X)(S)
  $$ 
be the submonoid of effective $0$-cycles, i.e. relative $0$-cycles with positive coefficients in $z_0(X)(S)$. Then $z_0^{\eff }(X)$ is a sub-presheaf of monoids in $z_0(X)$, and the latter is the group completion of the former in the category of presheaves on $\Nor $, see the bottom of page 81 in \cite{SuslinVoevodsky}.

Now recall that $X$ is said to be AF over $k$, if any finite collection of points of $X$ is contained in an affine neighbourhood in $X$. For example, all quasi-projective varieties over $k$ are AF, see Prop. (A.1.3) in Paper I in \cite{RydhThesis}.

So, assume that $X$ is AF over $k$. Then the $d$-th symmetric group $\Sigma _d$ acts admissibly on the $d$-fold fibred product $X^d$, symmetric powers of $X$ exist, see \cite{RydhThesis}, and we obtain the free monoid 
  $$
  \coprod _{i=0}^{\infty }\Sym ^d(X)
  $$
of the variety $X$ as an object in the category of presheaves on $\Nor $.


If $\bcF $ is a presheaf of monoids on some category $\catC $, and if $M$ is a set-theoretical monoid, considered as a constant presheaf on $\catC $, let 
  $$
  \bcF \otimes _{\NN }M
  $$ 
be the section-wise tensor product of $\bcF $ and $M$ over the semiring of natural numbers $\NN $. If now $\NN [1/p]$ is the smallest submonoid containing $\NN $ and $1/p$ in $\QQ $, we have the section-wise localization 
  $$
  \bcF [1/p]=\bcF \otimes _{\NN }\NN [1/p]
  $$ 
of $\bcF $ at $p$ which inverts $p$ (see Appendix A in \cite{Anderson} for details on localization of monoids).

In \cite{SuslinVoevodsky} Suslin and Voevodsky constructed a homomorphism of presheaves of monoids
  $$
  \sv _X:z^{\eff }(X)\lra 
  \left (\coprod _{i=0}^{\infty }\Sym ^d(X)\right)
  \left[\frac{1}{p}\right]
  $$
on the category of integral normal schemes over $k$, and then proved the following important theorem.

\begin{theorem}
\label{SV-theorem}
For any field $k$, and for any scheme $X$ of finite type over $k$, the homomorphism $\sv _X$ is an isomorphism of commutative monoids in the category of presheaves on integral normal schemes over $k$.
\end{theorem}

\begin{pf}
See Theorem 6.8 on page 82 in \cite{SuslinVoevodsky}.
\end{pf}

Now, if $k$ is a field, and $X$ an algebraic scheme over, let
  $$
  Z_0(X)
  $$
be the group of $0$-cycles on $X$, i.e. the free abelian group generated by closed points on $X$. 

\begin{corollary}
\label{rateqP1path}
Let $k$ be an algebraically closed field, and let $X$ be an algebraic scheme over $k$. Let also
  $$
  A,B\in Z_0(X)
  $$
be two $0$-cycles on $X(k)$. Then $A$ is rationally equivalent to $B$ on $X$ if and only if there exist an integer $u\geq 0$, a regular morphism
 $$
 f:\PR ^1\to \Sym ^{d+u}(X)
 $$
over $k$, and a positive $0$-cycle $C$ of degree $u$ on $X$, such that $A+C$ and $B+C$ are positive $0$-cycles of degree $d+u$ on $X$, 
  $$
  \sv _X(A+C)=f(0)
  $$
and
  $$
  \sv _X(B+C)=f(\infty )
  $$
\end{corollary}

\begin{pf}
This is a straightforward consequence of Theorem \ref{SV-theorem}. A much more general result is proven in \cite{Anderson}, Proposition 6.3.28, page 194
\end{pf}

Let
  $$
  \sigma :k\stackrel{\sim }{\lra }k'
  $$
be an isomorphism of fields, let $X$ be a scheme of finite type over $k$, and let $X'$ be a scheme given by the cartesian square 
  $$
  \diagram
  X' \ar[dd]^-{\varkappa } \ar[rr]^-{} & & \Spec (k') 
  \ar[dd]^-{\Spec (\sigma )} \\ \\
  X \ar[rr]^-{} & & \Spec (k)
  \enddiagram
  $$
in the category of schemes over $k$. The isomorphism of schemes $\tilde \sigma $ induces an isomorphism
  $$
  Z_0(\varkappa ):Z_0(X')\stackrel{\sim }{\lra }Z_0(X)
  $$
between the group of $0$-cycles on the scheme $X'$ over $k'$ and the group of $0$-cycles on the scheme $X$ over $k$.

\begin{corollary}
\label{fieldtwist}
The isomorphism $Z_0(\varkappa )$ and its inverse preserve the algebraic and rational equivalence relations of $0$-cycles on $X'$ over $k'$ and $X$ over $k$.
\end{corollary}

\begin{pf}
The isomorphism $\sigma $ induces an equivalence
  $$
  \Nor /k\stackrel{\sim }{\lra }\Nor /k'  
  $$
of the categories of integral normal schemes over $k$ and $k'$ respectively. This equivalence of categories induces naturally an equivalence 
  $$
  \PShv (\Nor /k)\stackrel{\sim }{\lra }\PShv (\Nor /k')
  $$
of the categories of set-valued presheaves on $\Nor /k$ and $\Nor /k'$, and the same for presheaves of monoids. 

The application of this equivalence to the Suslin-Voevodsky isomorphism of monoids gives us the commutative square  
  $$
  \diagram
  z^{\eff }(X'/k') \ar[dd]^-{\sim } \ar[rr]^-{\sv _{X'/k'}} & & 
  \left (\coprod _{i=0}^{\infty }\Sym ^d(X'/k')\right)[1/p]
  \ar[dd]^-{\sim } \\ \\
  z^{\eff }(X/k)  \ar[rr]^-{\sv _{X/k}} & & 
  \left (\coprod _{i=0}^{\infty }\Sym ^d(X/k)\right)[1/p]
  \enddiagram
 $$
in which the vertical arrows are isomorphisms of presheaves. The commutativity for the sections of the corresponding pre-sheaves on an algebraic curve $C$ over $k$ and its pull-back $C'=C\times_kk'$ over $k'$ gives us the first assertion of the lemma. If $C=\PR ^1$, we get the second one.
\end{pf}

\subsection{The group completion of the monoid of $0$-cycles}
\label{grpcompletion0ycles}

Let $k$ be an algebraically closed field, and let $X$ be an AF-scheme over $k$. We will be working with the free monoid
  $$
  \coprod _{d=0}^{\infty }\Sym ^d(X(k))
  $$
of $k$-rational points on $X$ and its group completion
  $$
  \left(\coprod _{d=0}^{\infty }\Sym ^d(X(k))\right)^+
  $$
as the quotient of its square by the diagonal, see Proposition 1.1 on page 70 in \cite{Weibel}. 

Then we have the standard homomorphisms
  $$
  \pi :\coprod _{d=0}^{\infty }\Sym ^d(X(k))\times \coprod _{d=0}^{\infty }\Sym ^d(X(k))\to \left(\coprod _{d=0}^{\infty }\Sym ^d(X(k))\right)^+\; ,
  $$
and
  $$
  \iota :\coprod _{d=0}^{\infty }\Sym ^d(X(k))\to 
  \left(\coprod _{d=0}^{\infty }\Sym ^d(X(k))\right)^+\; ,
  $$
see \cite{Weibel}. 

Since products of sets commute with coproducts, 
  $$
  \coprod _{d=0}^{\infty }\Sym ^d(X(k))\times \coprod _{d=0}^{\infty }\Sym ^d(X(k))=
  \coprod _{(d_1,d_2)\in \NN \times \NN }\left(\Sym ^{d_1}(X(k))\times 
  \Sym ^{d_2}(X(k))\right)
  $$
This gives us the map
  $$
  \pi _{d_1,d_2}:\Sym ^{d_1}(X(k))\times \Sym ^{d_2}(X(k))\to 
  \left(\coprod _{d=0}^{\infty }\Sym ^d(X(k))\right)^+\; ,
  $$
for each ordered pair of natural numbers $(d_1,d_2)$.

As above, lety
  $$
  Z_0(X(k))
  $$
be the group of $0$-cycles on $X(k)$, and consider the obvious homomorphism of monoids
  $$
  \alpha :\coprod _{d=0}^{\infty }\Sym ^d(X(k))\to Z_0(X(k))\; ,
  $$
sending a closed point on the symmetric power  
  $$
  A\in \Sym ^d(X(k))
  $$ 
to the corresponding effective $0$-cycle 
  $$
  A\in Z_0(X(k))\; .
  $$

The homomorphism $\alpha $ induces the homomorphism 
  $$
  \tau :\coprod _{d=0}^{\infty }\Sym ^d(X(k))\times 
  \coprod _{d=0}^{\infty }\Sym ^d(X(k))
  \to Z_0(X(k))\; ,
  $$
and, by distributivity, the map
  $$
  \tau _{d_1,d_2}:\Sym ^{d_1}(X(k))\times \Sym ^{d_2}(X(k))\to Z_0(X(k))\; ,
  $$
for each $d_1$ and $d_2$, sending
  $$
  (A,B)\mapsto A-B\; ,
  $$
where 
  $$
  A\in \Sym ^{d_1}(X(k))\qqand B\in \Sym ^{d_2}(X(k))
  $$
are two closed points on symmetric powers, and $A-B$ is the difference between the corresponding effective $0$-cycles in $Z_0(X(k))$. 

The homomorphism $\tau $ factorizes through the diagonal, inducing the unique homomorphism
  $$
  \theta :\left(\coprod _{d=0}^{\infty }\Sym ^d(X(k))\right)^+\lra \, Z_0(X(k))\; ,
  $$
sending
  $$
  [A,B]\mapsto A-B\; ,
  $$
and making the diagram
  $$
  \xymatrix{
  \coprod _{d=0}^{\infty }\Sym ^d(X(k)) \ar[rr]^-{\iota } \ar[ddrr]_-{\alpha } & &
  \left(\coprod _{d=0}^{\infty }\Sym ^d(X(k))\right)^+ \ar@{.>}[dd]^-{\hspace{+1mm}\exists !\theta } \\ \\
  & & Z_0(X(k))
  }
  $$ 
commute.

If 
  $$
  d_1=d_2\; ,
  $$
the, of course, 
  $$
  \deg (A-B)=0\; ,
  $$
and we obtain the map
  $$
  \tau _{d,d}:\Sym ^d(X(k))\times \Sym ^d(X(k))\to Z_0(X(k))_{\deg =0}
  $$
where 
  $$
  Z_0(X(k))_{\deg =0}
  $$ 
is the subgroup of degree $0$ $0$-cycles in $Z_0(X)$.

Let
  $$
  \cl _d:\Sym ^d(X(k))\times \Sym ^d(X(k))\to A_0(X)
  $$
be the composition of the map $\tau _{d,d}$ with the inclusion of $Z_0(X(k))$ in to $Z_0(X)$ and the quotient homomorphism from the group $Z_0(X)_{\deg =0}$ onto the Chow group $A_0(X)$ of degree $0$ $0$-cycles modulo rational equivalence on $X$. The map $\cl _d$ sends 
  $$
  (A,B)\mapsto [A-B]\; ,
  $$
where $[A-B]$ is the cycle class of $A-B$ in $A_0(X)$. 

\subsection{The proof of the Mumford-Ro\u \i tman's lemma}
\label{MRTproof}

Let $k$ be an algebraically closed field, and let $X$ be a smooth projective scheme over $k$. Our aim now is to give a rigorous proof of an arbitrary characteristic version of Mumford's Lemma 3 on page 201 in \cite{Mumford}. We will follow Roitman's approach given in \cite{RoitmanGamma}.

For short of notation, for any finite set of natural numbers
  $$
  \{ d_1,\ldots ,d_s\} 
  $$
let
  $$
  \Sym ^{d_1,\ldots ,d_s}(X)=
  \Sym ^{d_1}(X)\times \ldots \times \Sym ^{d_s}(X)
  $$
be the fibred product over the ground field $k$. In other words, 
$\{ d_1,\ldots ,d_s\} $ is a multi-degree of the corresponding multi-symmetric power of $X$. Then, of course, 
  $$
  \Sym ^{d_1,\ldots ,d_s}(X(k))=
  \Sym ^{d_1}(X(k))\times \ldots \times \Sym ^{d_s}(X(k))\; .
  $$

Since now let us assume that $X$ is a projective scheme over $k$. It follows that any symmetric power $\Sym ^d(X)$, and the product of symmetric powers are projective too. 

Let
  $$
  d\in \NN \; ,
  $$
and fix a closed immersion of the symmetric power $\Sym ^d(X)$ in to a projective space $\PR $ over $k$. Then the scheme 
  $$
  \MorS (\PR ^1,\Sym ^d(X))
  $$
admits the stratification 
  $$
  \MorS ^v(\PR ^1,\Sym ^d(X))=
  \coprod _{v\in \NN }\MorS ^v(\PR ^1,\Sym ^d(X))
  $$
by schemes parametrising morphisms of degree $v$ from $\PR ^1$ to $\Sym ^d(X)$, see \cite{Kollar}.

Choose two points $0$ and $\infty $ on $\PR ^1$, and let
  $$
  \ev (0,\infty ):\MorS ^v(\PR ^1,\Sym ^d(X))\to \Sym ^{d,d}(X)\ ,
  $$
be the evaluation morphism computing values of morphisms of degree $v$ from $\PR ^1$ to $\Sym ^d(X)$ at $0$ and $\infty $. This is a regular morphism of quasi-projective schemes over $k$ (see \cite{Kollar} or \cite{Debarre} for details on evaluation morphisms).

Let
  $$
  I^v_d\to \Sym ^{d,d}(X)
  $$
be the scheme-theoretic image of the morphism $\ev (0,\infty )$. 

\begin{proposition}
\label{MumfordRoitmanKey}
For any $v,d\in \NN $,
  $$
  I^v_d(k)\subset \cl _d^{-1}(0)\; .
  $$
\end{proposition}

\begin{pf}
The scheme-theoretic image $I^v_d$ is a closed subscheme in the symmetric power, which is nothing but the Zariski closure of the set-theoretic image 
  $$
  E^v_d\subset \Sym ^{d,d}(X)
  $$ 
of the morphism $\ev (0,\infty )$, with the induced reduced closed subscheme structure on it, see Tag 056B in \cite{StacksProject}. 

In particular, $I^v_d(k)$ is a subset in the set
    $$
    \Sym ^{d,d}(X)(k)=\Sym ^d(X(k))\times \Sym ^d(X(k))\; .
    $$

If a point
   $$
   (A,B)\in I^v_d(k)
   $$
is in $E^v_d$, then, of course, $A$ is rationally equivalent to $B$, and then $(A,B)$ is in $\cl _d^{-1}(0)$. 

Suppose 
  $$
  (A,B)\in I^v_d(k)\smallsetminus E^v_d\; .
  $$ 

A morphism of quasi-projective varieties is quasi-compact and locally of finite presentation. By Chevalley's theorem, Tag 054J in \cite{StacksProject}, the set-theoretical image $E^v_d$ of the morphism $\ev (0,\infty )$ is constructible, i.e. a finite union of locally closed subsets in the target variety $\Sym ^{d,d}(X)$. Let 
  $$
  I\subset I^v_d
  $$ 
be an irreducible component of the closed subscheme $I^v_d$ in $\Sym ^{d,d}(X)$ containing $(A,B)$ Embed $I$ in to a projective space, and use the Bertini theorem (Theorem 6.3 on pp 66 - 67 in \cite{Jouanolou}) to draw an irreducible curve 
  $$
  C\subset I
  $$ 
through the point $(A,B)$ in $I$. Let also
  $$
  C_0=C\cap E^v_d
  $$
be the set of points on the curve $C$ which are contained in the set-theoretic image of the map $\ev (0,\infty )$.

Choose also an irreducible curve 
  $$
  T_0\subset \MorS ^v(\PR ^1,\Sym ^d(X))\; ,
  $$
such that the closure of its set-theoretic image, under the map $\ev (0,\infty )$, is $C$, and let $T$ be the projective closure of $T_0$ over $k$. Let then $\tilde T$ be the normalization of $T$, and let $\tilde T_0$ be the pre-image of $T_0$ in $\tilde T$ under the normalization morphism $\tilde T\to T$. 

Notice that since the curve $\tilde T$ is smooth, the composition
  $$
  f:\tilde T\to T\dasharrow C
  $$
is regular. 

This all can be illustrated by the commutative diagram
  $$
  \diagram
  \tilde T \ar@/^3pc/[rrrr]^f \ar[rr]^-{} & & T \ar@{-->}[rr]^-{} & & C \ar[rr]^-{} & & I \ar[r]^-{} & I^v_d \ar[r]^-{} & \Sym ^{d,d}(X)\\ \\
  \tilde T_0 \ar[uu]_-{} \ar[rr]^-{} & & T_0 \ar[rr]^-{} \ar[uu]_-{} & &
  C_0 \ar[uu]_-{} \ar[rr]^-{} & & I\cap E^v_d \ar[uu]_-{} \ar[r]^-{} & E^v_d \ar[uu]_-{}
  \enddiagram
  $$

The regular morphism
  $$
  g_0:S_0=\tilde T_0\times \PR ^1\to T_0\times \PR ^1\to 
  \MorS ^v(\PR ^1,\Sym ^d(X))\times \PR ^1
  \stackrel{\ev }{\lra }\Sym ^d(X)  
  $$
defines a rational map
  $$
  g:S=\tilde T\times \PR ^1\dasharrow \Sym ^d(X)\; ,
  $$
where $\ev $ is the evaluation morphism $\ev ^v$. 

Again, since $\tilde T$ is smooth, the restrictions $g|_{\tilde T\times 0}$ and $g|_{\tilde T\times \infty }$ are regular surjective morphisms from $\tilde T$ to $C$.

If  
  $$
  P\in \tilde T
  $$ 
is a point mapped to $(A,B)$ under the regular morphism $f$, we obtain that
  $$
  g|_{\tilde T\times 0}(P)=A\qand g|_{\tilde T\times \infty }(P)=B\; .
  $$

 The determinacy locus of the rational map $g$ is a finite collection of rational curves $Q\times \PR ^1$ on $S$, where $Q$ runs through the points in $\tilde T\smallsetminus \tilde T_0$. Blowing up $S$ at points supported on the curves $Q\times \PR ^1$, we can resolve the indeterminacy locus and obtain a regular morphism
  $$
  \tilde g:\tilde S\to \Sym ^d(X)\; ,
  $$
such that both points $A$ and $B$ are in the set-theoretic image of $\tilde g$ and connected by a chain of the images of rational curves on $\tilde S$. 

It follows that $A$ is rationally equivalent to $B$ on $X$. That is, 
  $$
  (A,B)\in \cl _d^{-1}(0)\; ,
  $$
as required.
\end{pf}

Let 
  $$
  s:\Sym ^{d,d,u,u}(X)\to \Sym ^{d+u,d+u}(X)
  $$
be the morphism of schemes acting on $k$-points by the formula
  $$
  s(A,B,C)=(A+C,B+C)\; ,
  $$
and define a closed subscheme $W^{u,v}_d$ in $\Sym ^{d,d,u,u}(X)$ as the pullback of the closed subscheme $I^v_{d+u}$ under the morphism $s$, i.e. the diagram
  $$
  \diagram
  W_d^{u,v}\ar[rr]^-{} \ar[dd]^-{} & &
  \Sym ^{d,d,u}(X) \ar[dd]^-{s} \\ \\
  I^v_{d+u}\ar[rr]^-{i} & & \Sym ^{d+u,d+u}(X)
  \enddiagram
  $$
is a pullback square. Let also
  $$
  V^{u,v}_d=\pr _{1,2}(W^{u,v}_d)
  $$
to be the scheme-theoretic image of the closed subscmeme $W^{u,v}_d$ under the projection 
  $$
  \pr _{1,2}:\Sym ^{d,d,u}(X)\to \Sym ^{d,d}(X)\; .
  $$
Then $V^{u,v}_d$ is a closed subscheme in $\Sym ^d(X)$, for all triples $u,v,d$ in $\NN $.

The following result, the Mumford-Roitman lemma, is a straightforward consequence of Corollary \ref{rateqP1path} and Proposition \ref{MumfordRoitmanKey}.

\begin{corollary}
\label{MumfordRoitmanKey2}
  $$
  \cl _d^{-1}(0)=\cup _{u,v\in \NN }V_d^{u,v}(k)\; .
  $$
\end{corollary}

\begin{pf}
Let $A$ and $B$ be two $0$-cycles on $X(k)$, such that $(A,B)$ is in $\cl _d^{-1}(0)$, i.e. $A$ is rationally equivalent to $B$ on $X$. By Corollary \ref{rateqP1path}, there exist an integer $u\geq 0$, a regular morphism
 $$
 f:\PR ^1\to \Sym ^{d+u}(X)
 $$
over $k$, and a positive $0$-cycle $C$ of degree $u$ on $X(k)$, such that $A+C$ and $B+C$ are positive $0$-cycles of degree $d+u$ on $X(k)$, 
  $$
  \sv _X(A+C)=f(0)
  $$
and
  $$
  \sv _X(B+C)=f(\infty )\; .
  $$
Let $v$ be the degree of the morphism $f$. Then 
  $$
  f\in I^v_{d+u}(k)
  $$
and
  $$
  i(f)=s(A,B,C)\; .
  $$
It follows that 
  $$
  (A,B)\in V^{u,v}_d(k)\; .
  $$

Conversely, a $k$-rational point in $V^{u,v}_d$ is a point
  $$
  (A,B)\in \Sym ^{d,d}(X)(k)\; ,
  $$
such that there exist a positive $0$-cycle of degree $u$ on $X$ and 
  $$
  s(A,B,C)\in I^v_{d+u}(k)\; .
  $$
Then 
  $$
  (A+C,B+C)\in \cl _{d+u}^{-1}(0)
  $$
by Proposition \ref{MumfordRoitmanKey}. In other words, $A+C$ is rationally equivalent to $B+C$, and hence $A$ is rationally equivalent to $B$ on $X$, i.e. 
  $$
  (A,B)\in \cl _d^{-1}(0)\; .
  $$
\end{pf}

\section{General structure of Gysin kernels}
\label{gys}

\subsection{The Albanese morphism and the variety $A_1$}
\label{A1}

Let $k$ be an algebraically closed field, and let 
  $$
  l\neq \cha (k)
  $$
be a prime number, different from the characteristic of the field $k$. 

Let $A$ be an abelian variety over $k$. Let $A_{l^n}$ be the group of $l^n$-torsion points on $A$, let
  $$
  T_l(A)=\lim _n A_{l^n}
  $$
be the Tate module of $A$, with regard to $l$, and let
  $$
  V_l(A)=T_l(A)\otimes_{\ZZ _l}\QQ _l
  $$
be the corresponding vector space over $\QQ _l$.

It is convenient to work with $l$-adic homology and cohomology, and the duality between them, as studied in \cite{Laumon}. Let $a$ be the dimension of the abelian variety $A$. Laumon's comparison theorem, see page 170 in \cite{Laumon}, gives
  $$
  H_1(A,\QQ _l)\cong H^{2a-1}(A,\QQ _l(a))\; .
  $$
By Poincar\'e duality,
  $$
  H^{2a-1}(A,\QQ _l(a))\simeq H^1(A,\QQ _l)^{\vee }\; ,
  $$
where $V^{\vee }$ is always the vector space dual to a vector space $V$. 

On the other hand, it is well-known that
  $$
    H^1(A,\QQ _l)\simeq V_l(A)^{\vee }\; ,
  $$
see, for example, Theorem 12.1 on page 55 in \cite{MilneLEC}. Therefore,
  $$
  H_1(A,\QQ _l)\simeq H^1(A,\QQ _l)^{\vee }\simeq 
  (V_l(A)^{\vee })^{\vee }\simeq V_l(A)\; .
  $$

Now, let
  $$
  \alpha :A\to B
  $$
be a surjective homomorphism of abelian varieties over $k$, let 
  $$
  A_1=\ker (\alpha )^0
  $$
be the connected component of $0$ in the kernel $\ker (\alpha )$, and let
  $$
  i_1:A_1\to A
  $$
be the closed immersion of the abelian variety $A_1$ in to the abelian variety $A$ over $k$.

\begin{lemma}
\label{dualities}
The sequences
  $$
  0\to H_1(A_1,\QQ _l)\stackrel{{i_1}_*}{\lra }H_1(A,\QQ _l)
  \stackrel{\alpha _*}{\lra }H_1(B,\QQ _l)\to 0\; ,
  $$
  $$
  0\to H^1(B,\QQ _l)\stackrel{\alpha ^*}{\lra }H^1(A,\QQ _l)\stackrel{i_1^*}{\lra }
  H^1(A_1,\QQ _l)\to 0
  $$
are short exact.
\end{lemma}

\begin{pf}
The quotient-group
  $$
  A/A_1
  $$
is an abelian variety, see, for example, Section 9.5 in \cite{Polishchuk}, and the morphism $\alpha $ factors through the quotient as
  $$
  A\lra A/A_1\stackrel{\bar \alpha }{\lra }B\; .
  $$
Here $\bar \alpha $ is a regular surjective homomorphism with finite kernel
  $$
  \ker (\alpha )/A_1\; ,
  $$ 
hence an isogeny of abelian varieties over $k$.

Let
  $$
  C=A/A_1
  $$
be the quotient abelian variety, and consider the short exact sequence of abelian varieties
  $$
  0\to A_1\lra A\lra C\to 0
  $$
over $k$. It induces a short exact sequence of Tate modules
  $$
  0\to T_l(A_1)\lra T_l(A)\lra T_l(C)\to 0\; ,
  $$
and, tensoring with $\QQ _l$, we also obtain a short exact sequence
  $$
  0\to V_l(A_1)\lra V_l(A)\lra V_l(C)\to 0
  $$
of the corresponding vector spaces over $\QQ _l$.

Since $\bar \alpha :C\to B$ is an isogeny, it induces an isomorphism on $\QQ _l$-vector spaces, and we obtain a short exact sequence of $\QQ _l$-vector spaces
  $$
  0\to V_l(A_1)\stackrel{{i_1}_*}{\lra }V_l(A)
  \stackrel{\alpha _*}{\lra }V_l(B)\to 0\; .
  $$
 
And since first $l$-adic homology are isomorphic to $\QQ _l$-localized Tate modules, the latter short exact sequence is isomorphic to the short exact sequence
  $$
  0\to H_1(A_1,\QQ _l)\stackrel{{i_1}_*}{\lra }H_1(A,\QQ _l)
  \stackrel{\alpha _*}{\lra }H_1(B,\QQ _l)\to 0\; .
  $$

Dualizing, we obtain a short exact sequence
  $$
  0\to H_1(B,\QQ _l)^{\vee }\stackrel{\alpha _*^{\vee }}{\lra }
  H_1(A,\QQ _l)^{\vee }\stackrel{{i_1}_*^{\vee }}{\lra }  
  H_1(A_1,\QQ _l)^{\vee }\to 0\; .
  $$

Moreover, Laumon's duality for proper schemes, see page 170 in \cite{Laumon}, or, equivalently, Poincar\'e duality in smooth proper case, identifies
  $$
  H_1(A,\QQ _l)^{\vee }\simeq H^1(A,\QQ _l)\; ,
  $$
and the same for the abelian varieties $B$ and $A_1$. Under these identifications, the dual homomorphisms $\alpha _*^{\vee }$ and ${i_1}_*^{\vee }$ are the usual pullbacks $\alpha ^*$ and $i_1^*$ on cohomology. Therefore, we also obtain the short  exact sequence
  $$
  0\to H^1(B,\QQ _l)\stackrel{\alpha ^*}{\lra }H^1(A,\QQ _l)\stackrel{i_1^*}{\lra }
  H^1(A_1,\QQ _l)\to 0\; .
  $$
\end{pf}

\medskip

Next, choose an ample bundle $L$ on $A$. Its restriction
  $$
  L_1=L|_{A_1}
  $$
is ample on $K^0$ by Proposition 4.1 on page 23 in \cite{HartshorneAmple}. Then we obtai the standard polarizations 
  $$
  \lambda _L:A\to A^{\vee }
  $$
and 
  $$
  \lambda _{L_1}:A_1\to A_1^{\vee }\; ,
  $$
where $A^{\vee }$ and $A_1^{\vee }$ are abelian varieties dual to $A$ and $A_1$ respectively. 

Since $\lambda _L$ is an isogeny, and we work with coefficients in $\QQ _l$, the morphism $\lambda _L$ induces an isomorphism
  $$
  V_l(A)\stackrel{\sim }{\lra } V_l(A^{\vee })\; .
  $$

The standard Weil pairing, see Chapter I, Section 13 in \cite{MilneAV}, 
  $$
  V_l(A)_\times V_l(A^{\vee })_\to \QQ _l(1)
  $$
is perfect, and gives us an identification
  $$
  V_l(A^{\vee })\simeq V_l(A)^{\vee }(1)\; .
  $$
Then
  $$
  V_l(A)\stackrel{\sim }{\lra }V_l(A)^{\vee }(1)\; .
  $$

Applying this isomorphism togethjer with the standard identification
  $$
  H^1(A,\QQ _l)\simeq V_l(A)^{\vee }\; ,
  $$
see Theorem 12.1 on page 55 in \cite{MilneAV}, we obtain an isomorphism
  $$
  V_l(A)\simeq H^1(A,\QQ _l(1))\; .
  $$

By Laumon's comparison,
  $$
  H_1(A,\QQ _l)\simeq H^{2a-1}(A,\QQ _l(a))\; ,
  $$
see the bottom of page 172 in \cite{Laumon}. By Poincar\'e duality,
  $$
  H^{2a-1}(A,\QQ _l(a))\simeq H^1(A,\QQ _l)^{\vee }\; .
  $$
Therefore,
  $$
  H_1(A,\QQ _l)\simeq H^1(A,\QQ _l)^{\vee }\; ,
  $$
and hence
  $$
  H_1(A,\QQ _l)\simeq V_l(A)^{\vee \vee }\simeq V_l(A)\; .
  $$

As a result, 
  $$
  H_1(A,\QQ _l)\simeq H^1(A,\QQ _l(1))\; ,
  $$
or, equivalently,
  $$
  H^1(A,\QQ _l)\simeq H_1(A,\QQ _l(-1))\; 
  $$

Similarly, using the ample bundle $L_1$ we construct an isomorphism
  $$
  H^1(A_1,\QQ _l)\simeq H_1(A_1,\QQ _l(-1))\; 
  $$

Define a homomorphism $\zeta _1$ as the composition
  $$
  \diagram
  H^1(A_1,\QQ _l) \ar[dd]_-{\sim } \ar[rr]^-{\zeta _1} & & 
  H^1(A,\QQ _l) \\ \\
  H_1(A_1,\QQ _l(-1)) \ar[rr]^-{{i_1}_*} & & H_1(A,\QQ _l(-1)) 
  \ar[uu]^-{\sim } 
  \enddiagram
  $$
to be used in the next sections. 

\begin{lemma}
\label{zetainj}
The homomorphism $\zeta _1$ is injective. 
\end{lemma}

\begin{pf}
Since the vertical arrows in the latter commutative diagram are isomorphisms, and $i_*$ is injective by Lemma \ref{dualities}, the result follows.
\end{pf}

\medskip

We want to apply these general considerations in a situation when $B$ is the Albanese variety of a smooth projective surface, and $A$ is the Jacobian of a smooth projective curve inside it. 

Let first $X$ be a smooth projective variety over $k$, let 
  $$
  \Pic (X)=\Pic _{X/k}
  $$
be the Picard scheme of the variety $X$, let
  $$
  \Pic ^0(X)=\Pic ^0_{X/k}
  $$
be the connected component containing $0$ of the scheme $\Pic (X)$, and let
  $$
  \Alb (X)=\Alb _{X/k}
  $$ 
be the Albanese variety of $X$, i.e. the abelian variety dual to the Picard variety $\Pic ^0(X)$.

Fixing a closed point on $X$, gives us a universal Albanese morphism 
  $$
  a_X:X\to \Alb (X)
  $$
over $k$, see Chapter II, Section 3 in \cite{LangAV}. 

The following lemma is well-known, but we recall a proof for the convenience of the reader.

\begin{lemma}
\label{H1Albanese}
The morphism $a_X$ induces an isomorphism on $l$-adic first homology groups,
  $$
  {a_X}_*:H_1(X,\QQ _l)\stackrel{\sim }{\lra }H_1(\Alb _{X/k},\QQ _l)\; .
  $$
\end{lemma}

\begin{proof}
By Laumon's duality, see p 170 in \cite{Laumon}, and the Poincar\'e duality, it is equivalent to show that the pullback homomorphism
  $$
  a_X^*:H^1(\Alb _{X/k},\QQ _l)\to H^1(X,\QQ _l)
$$
is an isomorphism.

Since the Albanese variety is dual to the Picard variety, the Weil pairing yields an isomorphism
  $$
  H^1(\Alb (X),\QQ _l)\simeq V_l(\Pic ^0(X))(-1)\; .
  $$

Let $\NS (X)$ be the N\'eron-Severi group of the variety $X$, and consider the following commutative diagram
  $$
  \diagram
  0 \ar[r]^-{} & \Pic ^0(X) \ar[dd]^-{l^n} \ar[rr]^-{} & & \Pic (X) 
  \ar[dd]^-{l^n} \ar[rr]^-{} & & 
  \NS (X) \ar[dd]^-{l^n} \ar[r]^-{} & 0 \\ \\
  0 \ar[r]^-{} & \Pic ^0(X)  \ar[rr]^-{} & & \Pic (X) 
  \ar[rr]^-{}  & &
  \NS (X) \ar[r]^-{} & 0
  \enddiagram
  $$  

As $\Pic ^0(X)$ is an abelian variety over $k$, and the field $k$ is algebraically closed, the group of $k$-rational points on $\Pic ^0(X)$ is divisible. By the Snake Lemma, we get a short exact sequence
  $$
  0\to \Pic ^0(X)_{l^n}\to {\Pic ^0(X)}_{l^n}\to \NS (X)_{l^n}\to 0\; .
  $$

The multiplication by $l$ homomorphism
  $$
  l:\Pic ^0(X)_{l^{n+1}}\to \Pic ^0(X)_{l^n}
  $$
is surjective. Passing to the inverse limit gives us the short exact sequence of rationalized Tate modules
  $$
  0\to V_l(\Pic ^0(X))\to V_l(\Pic ^0(X))\to V_l(\NS (X))\to 0\; .
  $$

The short exact sequence of \'etale sheaves

  $$
  0\to \mu _{\l^n}\to \Gm \stackrel{l^n}{\lra }\Gm \to 0\; .
  $$
yields the isomorphism
  $$
  H^1(X,\QQ _l(1))\simeq V_l(\Pic (X))\; .
  $$

Therefore, we obtain the short exact sequence
  $$
  0\to H^1(\Alb (X),\QQ _l)\to H^1(X,\QQ _l)\to V_l(\NS (X))(-1)\to 0\; .
  $$

Since $\mathrm{NS}(X)$ is a finitely generated abelian group, its Tate module is finite and $V_l(\NS (X))$ vanishes, whence the result.
\end{proof}

Now assume that $X$ is a smooth projective surface over $k$. Fix a closed immersion of $X$ in to a projective space $\PR $ over $k$, and let $C$ be a smooth hyperplane section of $X$ in $\PR $. Let also
  $$
  \iota :C\to X
  $$
be the closed immersion of the curve $C$ in to $X$. 

The standard theta-divisor on the Jacobian 
  $$
  \Jac _{C/k}=\Pic ^0_{C/k}
  $$ yields a principle polarization which allows us to identify the Jacobian with its dual variety
  $$
  \Alb _{C/k}\; .
  $$
For short of notation, let
  $$
  A=\Alb _{C/k}
  $$
and 
  $$
  B=\Alb _{X/k}\; .
  $$
The closed immersion $\iota $ induces a homomorphism of abelian varieties 
  $$
  \alpha :A\to B
  $$
over $k$. 

Since $C$ is a hyperplane section of $X$ in $\PR $, by Lefschetz theorem, the homomorphism 
  $$
  \iota _*:H_1(C,\ZZ _l)\to H_1(X,\ZZ _l)
  $$
is surjective. By Lemma \ref{H1Albanese}, it follows that the homomorphism $\alpha $ is a surjective regular homomorphism of abelian varieties over the ground field $k$. 

On the other hand, the morphism $a_X$ factorizes through rational equivalence of $0$-cycles, and induces the Albanese homomorphism
  $$
  \alb _X:A_0(X)\to \Alb _{X/k}(k)
  $$
from the group of degree $0$ zero-cycles modulo rational equivalence to the group of $k$-rational points of the Albanese variety of the surface $X$. And we also have the Gysin homomorphism
  $$
  \iota _*:A(k)\to A_0(X)\; ,
  $$
where $A$ is the Jacobian of $C$.

As we mentioned in Introduction, the Gysin kernel 
  $$
  \Gys _0(k)=\ker (\iota _*)
  $$ 
is then contained in the bigger kernel
  $$
  \ker (\alb _X\circ \iota _*)
  $$
of the composition of the Gysin homomorphism $\iota _*$ on Chow groups of $0$-cycles and the Albanese homomorphism $\alb _X$.

Let
  $$
  G_1=\ker (\alpha )
  $$
be the kernel of the surjective regular homomorphism $\alpha $ of abelian varieties over $k$. Then $G_1$ is a group scheme, which is a closed subscheme in the Jacobian $A$, and the kernel of the composition $\alb \circ \iota _*$ is the group of $k$-rational points in $G_1$. 

Thus, the kernel 
  $$
  \Gys _1(k)=\ker (\alb \circ \iota _*)
  $$ 
is the set of $k$-rational points of the subscheme $G_1$ in $A$. 

Let $A_1$ be the connected component of the group scheme $G_1$ containing $0$. Then $A_1$ is an abelian subvariety in $A$,
  $$
  A_1(k)\subset \Gys _1(k)\; ,
  $$
and we can apply all the abstract Lemma \ref{dualities} and Lemma \ref{zetainj} to the connected component $A_1$ of the kernel of the homomorphism from the Jacobian $A$ of the curve $C$ to the Albanese variety $B$ of the surface $X$ over $k$.

\subsection{Countability lemmas over an uncountable domain}
\label{CL}

For any scheme $X$ let $|X|$ be the topological space underlaying the scheme $X$. If $Z$ is a closed subscheme in $X$, then $|Z|$ is called the support of $Z$ in $X$. 

Let $k$ be a field, and let $X$ be a scheme over $k$. For any subset
  $$
  S\subset |X|
  $$
let
   $$
   S(k)=S\cap X(k)
   $$
be the set of $k$-rational points in the set $S$. If
  $$
  f:X\to Y
  $$
is a morphism of schemes over $k$, and $S\subset |X|$, write
  $$
  f(S)\subset |Y|
  $$
for the set-theoretic image of $S$ under the morphism $f$. 


\begin{lemma}
\label{trivia}
Let $f: X \to Y$ be a morphism of schemes over a algebraically closed field $k$. Then
  $$
  f(X(k)) = f(X)(k)\; .
  $$	
\end{lemma}

\begin{pf}
By definition, $f(X)(k)$ is the intersection $f(X)\cap Y(k)$ of two subsets in $|Y|$. Obviously, $f(X(k))$ is a subset in $f(X)\cap Y(k)$. Conversely, let $P$ be a point in $f(X)\cap Y(k)$. The scheme-theoretic fibre $X_P$ of the morphism $f$ at $P$ is a non-empty closed subscheme in $X$ over $k$. Since $k$ is algebraically closed, $X_P(k)\neq \emptyset $. Therefore, if $Q\in X_P(k)$, then $Q\in X(k)$ and $f(Q)=P$.
\end{pf}

We will use the following terminology. In any topological space $T$, a $c$-closed subset in $T$ is the union of a countable collection of closed subsets in $T$. Accordingly, a $c$-open subset is the complement to a $c$-closed subset in $T$. 

\begin{lemma}
\label{syroezhka1}
Let $k$ be an uncountable algebraically closed field, and let $X$ be a scheme of finite type over $k$. Assume there exists a countable collection of Zariski closed subsets $Z_0,Z_1,Z_2,\dots $ in $X$, such that
  $$
  X(k)=\cup _{i\in \NN }Z_i(k)\; .
  $$
Then there exists $i_0\in \NN $, such that
  $$
  X=Z_{i_0}\; .
  $$
\end{lemma}

\begin{pf}
The proof can be easily reduced to the case when $X$ is integral and affine. 

Assume first that $X$ is the affine space $\AF ^n$ over $k$, and perform induction by the dimension $n$. If $n=0$, the lemma is obvious. Suppose $n>0$ and consider a hyperplane $H$ in $\AF ^n$. Then
  $$
  H(k)=\cup _{i\in \NN }(H\cap Z_i)(k)\; ,
  $$
and by the inductive hypothesis there exists $i_0$, such that 
  $$
  H(k)=(H\cap Z_{i_0})(k)\; .
  $$
It follows that
  $$
  H\subset Z_{i_0}\; .
  $$
Then either $H=Z_{i_0}$ or $Z_{i_0}$ is the whole space $\AF ^n$. Since the ground field $k$ is uncountable, the set of hyperplanes in $\AF ^n$ is uncountable too. And as the set of closed subschemes $Z_i$ in $\AF ^n$ is countable by assumption, $Z_{i_0}$ must be $\AF ^n$ for at least one index $i_0$. 

If $X$ is not an affine space, by Noether normalization, there exists a finite surjective morphism of schemes 
  $$
  f:X\to \AF ^n\; ,
  $$
where $n$ is the dimension of $X$. By Lemma \ref{trivia},
  $$
  \AF ^n(k)=f(X)(k)=f(X(k))=f(\cup _{i\in \NN }Z_i(k))=
  $$
  $$
  \cup _{i\in \NN }f(Z_i(k))=\cup _{i\in \NN }f(Z_i)(k)\; .
  $$ 
Since a finite morphism is proper, each $f(Z_i)$ is Zariski closed in $\AF ^n$. As we proved the lemma for affine spaces, there exists $i_0$ such that
  $$
  \AF ^n(k)=f(Z_{i_0})(k)\; .
  $$
Then 
  $$
  \dim (Z_{i_0})=n\; ,
  $$
and hence 
  $$
  X=Z_{i_0}\; .
  $$
\end{pf}

Recall that a noetherian scheme is a noetherian topological space, and any subspace of a noetherian topological space is noetherian. Moreover, if $X$ is a noetherian space, for any closed subset $Z$ in $X$, there exists a finite collection of irreducible closed subsets $Z_0,\ldots ,Z_n$ in $X$, such that $X$ is the union $Z_0\cup \ldots \cup Z_n$, and if we drop inclusions $Z_i\subset Z_j$, such a decomposition is unique up to re-indexing of the subsets, see, for example, Proposition 1.5 on page 5 in \cite{Hartshorne}. We need a similar theory for countable unions of closed subsets in a noetherian space $X$.

Namely, a subset $Z$ in $X$ will be called $c$-closed, if there exists a countable collection of closed subsets   
  $$
  Z_0,Z_1,Z_2,\ldots \subset |X|\; ,
  $$
such that 
  \begin{equation}
  \label{irred}
  Z=\cup _{i\in \NN }Z_i\; .   	
  \end{equation}
If each $Z_i$ is irreducible, we will say that (\ref{irred}) is an irreducible decomposition of the $c$-closed subset $Z$ in $X$. If, moreover,
  $$
  Z_i\not \subset Z_j
  $$
whenever 
  $$
  i\neq j\; ,
  $$
we will say that (\ref{irred}) is an irredundant irreducible decomposition of the $c$-closed subset $Z$ in $X$. 


\begin{lemma}
\label{syroezhka2}
Let $k$ be an uncountable algebraically closed field, and let $X$ be a scheme of finite type over $k$. Then any $c$-closed subset in $|X|$ admits an irredundant irreducible decomposition, unique up to the re-indexing of subsets in it.
\end{lemma}

\begin{pf}
Let $Z$ be a $c$-closed subset in $|X|$, and let
  $$
  Z=\cup _{i\in \NN }Z_i
  $$
be a representation of $Z$ as a union of Zariski closed subsets, which exists by the definition of a $c$-closed subset in a noetherian space.  

For each index $i$, let 
  $$
  Z_i=\cup _jZ_{ij}
  $$
be the decomposition of $Z_i$ in to its irreducible components inside $X$. Dropping $Z_{ij}$, such that $Z_{ij}\subset Z_{i'j'}$ for some $i'$ and $j'$, we obtain an irredundant irreducible decomposition of $Z$. We only need to show uniqueness. 

Suppose 
  $$
  Z=\cup _{i\in \NN }Z_i=\cup _{i\in \NN }Z'_i
  $$
are two irredundant irreducible decompositions of $Z$. For any index $i$,
  $$
  Z_i(k)=\cup _{j\in \NN }(Z_i\cap Z_j)(k)\; ,
  $$
and then
  $$
  Z_i\subset Z'_{j_0}
  $$
for some $j_0$ by Lemma \ref{syroezhka1} applied to $Z_i$ with the induced reduced closed subscheme structure on it. 

Similarly,
  $$
  Z'_{j_0}\subset Z_{i_0}
  $$
for some index $i_0$. Then 
  $$
  Z_i\subset Z_{i_0}\; ,
  $$
and as the decomposition of $Z$ in to the subsets $Z_i$ is irredundant irreducible, 
  $$
  Z_i=Z'_{j_0}\; .
  $$
This finishes the proof of uniqueness. 
\end{pf}

\begin{lemma}
\label{syroezhka3}
Let $k$ be an uncountable algebraically closed field, and let $A$ be an abelian variety over $k$. Assume there exists a $c$-closed subset $B$ in $|A|$, such that $B(k)$ is a subgroup in $A(k)$. Then the irredundant irreducible decomposition of $B$ contains a unique irreducible component $B_0$ passing through $0\in A$, and $B_0$ is an abelian subvariety in the abelian variety $A$. 
	\end{lemma}

\begin{pf}
Let 
  \begin{equation}
  \label{irred2}
  B=\cup _{i\in \NN }B_i
  \end{equation}
be an irredundant irreducible decomposition of the $c$-closed subset $B$, which exists by Lemma \ref{syroezhka2}. Since $B(k)$ is a subgroup in $A(k)$, we have that $0\in B(k)$. Renumbering the components in (\ref{irred2}), let 
  $$
  B_0,\ldots ,B_n
  $$ 
be all the irreducible components in (\ref{irred2}) which contain $0$, considered with the reduced induced closed subscheme structure on them. Consider the summation morphism
  $$
  B_0\times \ldots \times B_n\to A
  $$
induced by the group law in $A$. The product of irreducible schemes is irreducible, and the set-theoretic image of an irreducible set is irreducible. Therefore, the set-theoretic image 
  $$
  B_0+\ldots +B_n
  $$
of the above summation morphism is irreducible. As it is contained in $B$, we must have that
  $$
  B_0+\ldots +B_n=\cup _{i\in \NN }((B_0+\ldots +B_n)\cap B_i)
  $$
in $|A|$, thanks to the decomposition (\ref{irred2}). Then there must exists $i_0$, such that
  $$
  (B_0+\ldots +B_n)\cap B_{i_0}=B_0+\ldots +B_n
  $$
by Lemma \ref{syroezhka2}. It follows that 
  $$
  B_0+\ldots +B_n\subset B_{i_0}\; ,
  $$
whence 
  $$
  B_i\subset B_{i_0}
  $$
for all $i=0,\dots ,n$. Since (\ref{irred2}) is an irredundant irreducible decomposition of $B$, it follows that 
  $$
  n=0
  $$
and $B_0$ is the only irreducible component of $B$ passing through $0$ in $A$. 

If $B_0=\{ 0\} $, then $B_0$ is an abelian subvariety in $A$. Suppose $B_0$ is a set bigger than $\{ 0\} $ and consider $B_0$ with the induced reduced closed subscheme structure on it. Consider the subtraction morphism
   $$
   B_0\times B_0\to A
   $$
sending $(P,Q)$ to $P-Q$, for any two points $P$ and $Q$ in $B_0$. By the same arguments, the set-theoretic image $B_0-B_0$ of the subtraction morphism is irreducible and contained in $B_{i_0}$ for some index $i_0$, and then $B_0$ is a subset in $B_{i_0}$. Since (\ref{irred2}) is an irredundant, we obtain that $i_0=0$, and then
  $$
  B_0-B_0\subset B_0\; .
  $$
It follows that the closed subscheme $B_0$ is an abelian subvariety in $A$.
\end{pf}

In the same way, one can also prove 

\begin{lemma}
\label{syroezhka4}
Let $k$ be an uncountable algebraically closed field, and let $A$ be an abelian variety over $k$. Assume there exists a Zariski closed subset $B$ in $|A|$, such that $B(k)$ is a subgroup in $A(k)$. Then the irredundant irreducible decomposition of $B$ contains a unique irreducible component $B_0$ passing through $0\in A$, and $B_0$ is an abelian subvariety in the abelian variety $A$. 
	\end{lemma}

\begin{pf}
Similar to the proof of Lemma \ref{syroezhka3}.	
\end{pf}

\subsection{Gysin kernels and the variety $A_0$}
\label{A0}

Our next aim is to show that the Gysin kernel $G_0(k)$ is the union of a countable collection of shifts by elements in $A$ of a certain abelian subvariety $A_0$ in $A_1$. We first prove two easy lemmas, and then prove our first result, Theorem \ref{gysstructure}.

\begin{lemma}
\label{Riemann-Roch}
Let $k$ be an algebraically closed field, and let $C$ be a smooth projective curve of genus $g$ over $k$. The map
  $$
  \cl _d:\Sym ^d(C(k))\times \Sym ^d(C(k))\to A_0(C)
  $$
is surjective, if 
  $$
  d\geq g\; .
  $$
\end{lemma}

\begin{pf}
Let $P$ be a point in $C(k)$, and let $D$ be a divisor of degree $0$ on $C$. By the Riemann-Roch theorem,
  $$
  h^0(\bcO (D+n[P]))\geq 1-g+d\geq 1\; .
  $$
It follows that there exists an effective divisor $D'$ on $C$ rationally equivalent to $D+dP$ on $C$.

Then
  $$
  \deg (D')=\deg (D)+d=d\; ,
  $$
  $$
  \deg (dP)=d\; ,
  $$
and 
  $$
  \cl _d(D',dP)=[D]\; .
  $$
\end{pf}

Let $k$ be an algebraically closed field, and let 
  $$
  f:X\to Y
  $$
be a morphism of schemes of finite type over $k$. For any closed subscheme
  $$
  Z\to Y
  $$
we have the scheme-theoretic inverse image
  $$
  f^{-1}(Z)=X\times _YZ
  $$
of the closed subscheme $Z$, and the set-theoretic pre-image
  $$
  f^{-1}(Z(k))
  $$
of the set of $k$ rational points of the subscheme $Z$. 

\begin{lemma}
In terms above, if $f$ is finite, then
  $$
  f^{-1}(Z)(k)=f^{-1}(Z(k))\; .
  $$
\end{lemma}

\begin{pf}
Obviously, $f^{-1}(Z)(k)$ is a subset in $f^{-1}(Z(k))$. Let $P$ be a point in $f^{-1}(Z(k))$. Then $P$ is a point of the scheme $f^{-1}(Z)$, and $Q=f(P)$ is a point in $Z(k)$, and hence in $Y(k)$. Since $f$ is finite, it is quasi-finite. Then the residue field extension $\kappa (Q)\subset \kappa (P)$ is finite, see Tag 01TG in \cite{StacksProject}. Since $Q$ is $k$-rational, $\kappa (Q)$ is $k$, and as $k$ is algebraically closed, $\kappa (P)$ is $k$ too.
\end{pf}

Let $k$ be an algebraically closed field, and let $X$ be a smooth projective variety over $k$. Let
  $$
  CH_0(X)
  $$
be the Chow group of $0$-cycles, and let 
  $$
  A_0(X)=\ker (CH_0(X)\stackrel{\deg }{\xrightarrow{\hspace{6mm}}}\ZZ )
  $$
be the kernel of the degree homomorphism, or the Chow group of $0$-cycles of degree $0$ modulo rational equivalence on $X$.

Now again let $X$ is a smooth projective surface over $k$, let $C$ be a smooth projective curve inside $X$, and let
  $$
  \iota :C\to X
  $$
be the corresponding closed immersion over $k$. As above, let
  $$
  A=\Alb _{C/k}
  $$
be the Jacobian variety of the curve $C$, and let
  $$
  B=\Alb _{C/k}
  $$
be the Albanese variety of the surface $X$. The closed immersion $\iota $ induces the homomorphism 
  $$
  \alpha :A\to B
  $$
of abelian varieties over $k$, and the push-forward (Gysin) homomorphism 
  $$
  \iota _*:A_0(C)\to A_0(X)\; .
  $$

Fixing a point on $C$, the classical Abel-Jacobi theorem gives us the standard isomorphism of abelian groups
  $$
  \AJ :A_0(C)\stackrel{\sim }{\lra }A(k)\; ,
  $$
and we obtain the Gysin homomorphism 
  $$
  \iota _*:A(k)\to A_0(X)\; .
  $$

Notice that the homomorphism $\iota _*$ is a part of Gysin structure in terms of Bloch's high Chow groups as a ring cohomology theory, see Section 1.8 in \cite{PaninAEMHT} or Chapter 2 in \cite{PaninHHA}. This is why we call
  $$
  \Gys _0(k)=\ker (\iota _*)
  $$
the {\it Gysin kernel} associated to the smooth pair $(X,C)$.

If 
  $$
  k\subset K
  $$
is a field extension, then
  $$
  A_K=\Jac _{C/k}\times _kK
  $$
is the Jacobian of the curve
  $$
  C_K=C\times _kK\; .
  $$
Define 
  $$
  \Gys _0(K)=\ker ({\iota _K}_*)
  $$
to be the kernel of the obvious push-forward homomorphism
  $$
  {\iota _K}_*:A(K)\to A_0(X_K)\; ,
  $$
where
  $$
  \iota _K:C_K\to X_K
  $$
is the closed immersion over $K$. 

It is important to emphasise that, in contrast to $G_1(k)$, the group $\Gys _0(k)$ is not the set of $k$-rational points in a group subscheme in $A$. However, the following is true:

\begin{theorem}
\label{gysstructure}
Assume that the ground field $k$ is algebraically closed and uncountable. Then there exists an abelian variety $A_0$ in the Jacobian $A$ over $k$, and a countable subset $\Xi $ in $A(k)$, such that $\Gys _0(k)$ is the union of translates 
  $$
  P+A_0(k)
  $$
by all the points
  $$
  P\in \Xi 
  $$
in $A(k)$.
\end{theorem}

\begin{pf}
Let $g=g(C)$ be the genus of the curve $C$, and assume $d\geq g$. By Lemma \ref{Riemann-Roch}, the map 
  $$
  \cl _d:\Sym ^{d,d}(C(k))\to A_0(C)
  $$ 
is surjective. 

Let also
  $$
  \Sym ^{d,d}(\iota ):\Sym ^{d,d}(C)\to \Sym ^{d,d}(X)\; .
  $$
be the morphism of schemes induced by the closed immersion $\iota $ on symmetric powers, and for short of notation let 
  $$
  i:\Sym ^{d,d}(C(k))\to \Sym ^{d,d}(X(k))
  $$
be the map on $k$-rational points induced by the morphism $\Sym ^{d,d}(\iota )$.

The obvious commutative diagram
  $$
  \diagram
  \Sym ^{d,d}(C(k)) \ar[dd]^-{\cl _d} \ar[rr]^-{i} & & \Sym ^{d,d}(X(k)) 
  \ar[dd]^-{\cl _d } \\ \\
  A_0(C) \ar[dd]^-{\AJ } \ar[rr]^-{\iota _*} & & A_0(X) \ar[dd]^-{\id } \\ \\
  A(k) \ar[rr]^-{\iota _*} & & A_0(X)
  \enddiagram
  $$ 
tells then that
  $$
  \Gys _0(k)=\AJ (\cl _d(i^{-1}(\cl ^{-1}_d(0))))
  $$
 
By Corollary \ref{MumfordRoitmanKey2}, there exist closed subschemes 
  $$
  V_d^{u,v}\to \Sym ^{d,d}(X)
  $$
over $k$, for each pair of natural numbers $(u,v)$, such that the pre-image of $0$ under the right map $\cl _d$ is given by the formula
  $$
  \cl _d^{-1}(0)=\cup _{u,v\in \NN }V_d^{u,v}(k)\; .
  $$
Let
  $$
  Z_d^{u,v}\to \Sym ^{d,d}(C)
  $$
be the scheme-theoretic inverse image of the closed subscheme $V_d^{u,v}$ under the morphism $\Sym ^{d,d}(\iota )$.

The left vertical composition $\AJ \circ \cl _d$ in our commutative diagram is induced, on $k$-rational points, by the corresponding regular morphism of schemes
  $$
  \Sym ^{d,d}(C)\to A
  $$
over $k$. Since this morphism is proper, as a morphism of projective schemes, the scheme-theoretic images 
  $$
  A_d^{u,v}\to A
  $$
of the closed subschemes $Z_d^{u,v}$ under this morphism are the set-theoretic images with the induced reduced closed subscheme structure on them, and
  $$
  \Gys _0(k)=\cup _{u,v\in \NN }A_d^{u,v}(k)
  $$
in $A(k)$. 

Now, since the union 
  $$
  \cup _{u,v\in \NN }|A_d^{u,v}|
  $$
is a $c$-closed subset in $A$ whose set of $k$-rational points is $\Gys _0(k)$, by Lemma \ref{syroezhka2}, it admits a unique (up to reindexing) irredundant irreducible decomposition
  $$
  \cup _{u,v\in \NN }|A_d^{u,v}|=\cup _{i\in \NN }A_i\; ,
  $$
and then
  $$
  \Gys _0(k)=\cup _{i\in \NN }A_i(k)\; .
  $$

By Lemma \ref{syroezhka3}, this irredundant irreducible decomposition contains a unique irreducible component $A_0$ passing through $0\in A$, and $A_0$ is an abelian subvariety in the abelian variety $A$. 

As $0\in A$ and $A_0(k)$ is the Gysin kernel $\Gys _0(k)$ is a subgroup in $A(k)$,
  $$
  \Gys _0(k)=\cup _{P\in \Gys _0(k)}(P+A_0(k))\; .
  $$
Let
  $$
  \Xi \subset \Gys _0(k)
  $$
be the subset of points $P\in \Gys _0(k)$, such that $P+A_0(k)$ is not contained in $Q+A_0(k)$ for any other point $Q\in \Gys _0(k)$. Then 
  $$
  \Gys _0(k)=\cup _{P\in \Xi }(P+A_0(k))\; ,
  $$
and we only need to show that the set $\Xi $ countable. 

Since
  $$
  \cup _{i\in \NN }A_i(k)=\cup _{P\in \Xi }(P+A_0(k))\; ,
  $$
we have that
  $$
  P+A_0(k)=\cup _{i\in \NN }(A_i(k)\cap (P+A_0(k)))=
  \cup _{i\in \NN }((A_i\cap (P+A_0))(k))
  $$
for each point $P\in \Xi $.

As $A_0$ is an abelian subvariety in $A$, the scheme-theoretic shift $P+A_0$, i.e. the scheme-theoretic image of $A_0$ under the addition of $P$ morphism from $A$ to $A$, is a scheme of finite type over $k$, and 
  $$
  P+A_0(k)=(P+A_0)(k)\; .
  $$
We then obtain that 
  $$
  (P+A_0)(k)=\cup _{i\in \NN }((A_i\cap (P+A_0))(k))\; .
  $$
Applying Lemma \ref{syroezhka1}, we get  
  $$
  P+A_0=A_i\cap (P+A_0)
  $$
for some index $i$. It follows that
  $$
  P+A_0(k)\subset A_i(k)\; ,
  $$
or, equivalently,
  $$
  A_0(k)\subset -P+A_i(k)\; .
  $$ 
  
Similarly,
  $$
  -P+A_i(k)\subset A_j(k)
  $$
form some index $j$. 

Then
  $$
  A_0(k)\subset A_j(k)\; ,
  $$
and since our decomposition is irredundant,
  $$
  A_0=A_j\; .
  $$

Thus,
  $$
  A_0(k)\subset -P+A_i(k)\subset A_0(k)\; ,
  $$
and, therefore,
  $$
  P+A_0(k)=A_i(k)\; ,
  $$
for any point $P\in \Xi $. 

It follows that $\Xi $ is countable.
\end{pf} 

The following lemma is obvious.

\begin{lemma}
\label{supereasy}
Assuming the ground field $k$ is algebraically closed and uncountable, the abelian variety $A_0$ is contained in the abelian variety $A_1$.
\end{lemma}

\begin{pf}
Since $A_0$ is connected, it is a subvariety in a connected component of the group scheme $G_1$. And since $0\in A_0$, it follows that $A_0$ is in $A_1$.
\end{pf}

Notice also that, applying the general Lemma \ref{dualities} and Lemma \ref{zetainj} in Section \ref{A1} to the abelian subvariety $A_0$ in $A$, we can define an injective homomorphism $\zeta _0$ as the composition
  $$
  \diagram
  H^1(A_0,\QQ _l) \ar[dd]_-{\sim } \ar[rr]^-{\zeta _0} & & 
  H^1(A,\QQ _l) \\ \\
  H_1(A_0,\QQ _l(-1)) \ar[rr]^-{{i_0}_*} & & H_1(A,\QQ _l(-1)) 
  \ar[uu]^-{\sim } 
  \enddiagram
  $$
similar to the homomorphism $\zeta _1$. 

And, applying the same lemmas to the abelian subvariety $A_0$ in $A_1$, we also get an injective homomorphism $\zeta _{01}$ as the composition
  $$
  \diagram
  H^1(A_0,\QQ _l) \ar[dd]_-{\sim } \ar[rr]^-{\zeta _{01}} & & 
  H^1(A_1,\QQ _l) \\ \\
  H_1(A_0,\QQ _l(-1)) \ar[rr]^-{{i_{01}}_*} & & H_1(A,\QQ _l(-1)) 
  \ar[uu]^-{\sim } 
  \enddiagram
  $$
where $i_{01}$ is the closed immersion of $A_0$ in to $A_1$ over $k$. 

Clearly,
  $$
  \zeta _1\circ \zeta _{01}=\zeta _0\; .
  $$

\subsection{Very general versus geometric generic}
\label{gengeneric}

Let $k$ be an uncountable algebraically closed field, and let $S$ be an integral algebraic scheme over $k$. A $c$-closed subset in $S$ is a union of a countable collection of Zariski closed irreducible subsets in $S$. A $c$-open subset in $S$ is the complement to a $c$-closed subset in $S$. As usual, we will be saying that a property of points in $S$ holds for a very general point on $S$ if there exists a $c$-open subset $U$ in $S$, such that this property holds for each closed point in $U$.

Choose a countable algebraically closed subfield $k_0$ in $k$, such that there exists an integral model $S_0$ over $k_0$ with 
  $$
  S=S_0\times _{\Spec (k_0)}\Spec (k)\; .
  $$
For any closed subscheme $Z$ in $S_0$, let 
  $$
  i_Z:Z\subset S_0
  $$
be the corresponding closed immersion over $k$. Since the field $k_0$ is countable and the ideal of $Z$ is finitely generated, there exists only countably many closed subschemes $Z$ in $S_0$. For each $Z$ let 
  $$
  U_Z=S\smallsetminus \im (i_Z)
  $$
be the complement to the image of the closed immersion $i_Z$ in $S$, let 
  $$
  Z_k=Z\times _{\Spec (k_0)}\Spec (k)
  $$ 
and let 
  $$
  (U_Z)_k=U_Z\times _{\Spec (k_0)}\Spec (k)\; .
  $$
If 
  $$
  (i_Z)_k:Z_k\to S
  $$ 
is the scalar extension of $i_Z$ from $k_0$ to $k$, then
  $$
  (U_Z)_k=S\smallsetminus \im ((i_Z)_k)\; .
  $$
Let, furthermore,
  $$
  U=S\smallsetminus \cup _Z\im (({i_Z})_k)=
  \cap _Z(U_Z)_{k }\; ,
  $$
where the union is taken over closed subschemes $Z$, such that 
  $$
  \im (i_Z)\neq S_0\; .
  $$
The set $U$ is $c$-open in $S$.

For our purposes, without loss of generality, we may assume that $S_0$ is affine. Let then 
  $$
  k_0[S_0]=\Gamma (S_0,\bcO _{S_0})
  $$
be the ring of regular functions on $S_0$, and let
  $$
  k_0(S_0)=k_0[S_0]_{(0)}
  $$
be the function field on the affine scheme $S_0$. Similarly, let 
  $$
  k[S]
  $$
and 
  $$
  k(S)=k[S]_{(0)}
  $$
be the ring of regular functions and the field of rational functions on the scheme $S$. 
  
\begin{lemma}
\label{iso}
For any $k$-rational point $P$ in $U$, there exists a (non-canonical) isomorphism 
  $$
  e_P:\overline {k(S)}\stackrel{\sim }{\lra }k
  $$
over $k_0$, such that
  $$
  e_P(f)=f(P)
  $$
for any function $f$ in the ring $k_0[S_0]$.
\end{lemma}

\begin{pf}
Let 
  $$
  f_P:\Spec (k)\to S
  $$
be the morphism corresponding to the point $P$ over $k$. The image of $f_P(P)$ under the projection 
  $$
  \pi :S\to S_0
  $$ 
belongs to the set $U_Z$, for each closed subscheme $Z$ in $S_0$, such that $\im (i_Z)\neq S_0$. Then this image is noting but the generic point 
  $$
  \eta _0=\Spec (k_0(S_0))
  $$ 
of the scheme $S_0$. 

In other words, there exists a morphism 
  $$
  h_P:\Spec (k)\to \Spec (k_0(S_0))=\eta _0\; ,
  $$
such that 
  $$
  \pi \circ f_P=g_0\circ h_P\; ,
  $$
where $g_0$ is a morphism from the generic point $\eta _0$ of $S_0$ to $S_0$.

In terms of commutative rings, if 
  $$
  \ev _P:k[S]\to k
  $$ is the evaluation at $P$, i.e. the morphism inducing $f_P$ on spectra, there exists a homomorphism of fields $\epsilon _P$ making the diagram
  \begin{equation}
  \label{ulitka}
  \diagram
  k[S] \ar[rr]^-{\ev _P} & & k\\ \\
  k_0[S_0] \ar[uu]_-{} \ar[rr]^-{} & & k_0(S_0)
  \ar[uu]_-{\epsilon _P}
  \enddiagram
  \end{equation}
commutative, where $k$ in the top right corner is considered as the residue field of the scheme $S$ at $P$. Since the left vertical homomorphism is injective, the set 
  $$
  k _0[S_0]\smallsetminus \{ 0\} 
  $$ 
is a multiplicative system in the ring $k[S]$, and it is not hard to see that 
  $$
  (k_0[S_0]\smallsetminus \{ 0\} )^{-1}k[S]=k[S]\otimes _{k_0[S_0]}k_0(S_0)\; .
  $$ 
This is why there exists a unique universal homomorphism of rings 
  $$
  (k_0[S_0]\smallsetminus \{ 0\} )^{-1}k[S]\to k
  $$ 
whose restriction to the ring $k[S]$ is $\ev _P$ and the restriction to the field $k_0(S_0)$ is $\epsilon _P$. 

We want to construct an embedding of $k(S)$ in to $k$ whose restriction to $k_0(S_0)$ would be $\epsilon _P$. Such an embedding will not be over the ring $(k_0[S_0]\smallsetminus \{ 0\} )^{-1}k[S]$. 

So let $d$ be the dimension of $S_0$. By the Noether normalization lemma, there exist $d$ algebraically independent elements, $x_1,\dots ,x_d$ in $k_0[S_0]$, such that the latter ring is integral and finitely generated, and hence finite over the ring $k _0[x_1,\dots ,x_d]$, and $k_0(S_0)$ is algebraic over the field of fractions $k_0(x_1,\dots ,x_d)$, see $\S 1$ in Chapter II of Lang's book \cite{Lang}. Then the ring $k[S]$ is integral over the ring $k[x_1,\dots ,x_d]$, and the field $k(S)$ is algebraic over the field $k(x_1,\dots ,x_d)$. 

Let 
  $$
  b_i=\ev _P(x_i)
  $$ 
for all indices $1,\dots ,d$. Since $P\in U$, the elements $b_1,\dots ,b_d$ are algebraically independent over $k_0$. Extend the set $\{ b_1,\dots ,b_d\} $ to a transcendental basis $B$ of $k$ over $k_0$, 
  $$
  \{ b_1,\dots ,b_d\} \subset B\; ,
  $$
to have that 
  $$
  k=k_0(B)
  $$
over $k$. 

Since $B$ is of infinite cardinality, so is the set $B\smallsetminus \{ b_1,\dots ,b_d\} $. Choose and fix a bijection
  $$
  B\stackrel{\sim }{\to }
  B\smallsetminus \{ b_1,\dots ,b_d\} \; .
  $$
It gives a field embedding
  $$
  k=k_0(B)\simeq
  k_0(B\smallsetminus \{ b_1,\dots ,b_d\} )\subset
  k_0(B)
  $$
over $k_0$, such that the set $\{ b_1,\dots ,b_d\} $ is algebraically independent over its image. The latter embedding induces a new field embedding
  $$
  k (x_1,\dots ,x_d)\to k
  $$
sending $x_i$ to $b_i$ for each $i$. The restriction of this field embedding on $k_0(x_1,\dots ,x_d)$ is the restriction of $\epsilon _P$ on the same field. Since $k(S)$ is the tensor product of $k(x_1,\dots ,x_d)$ and $k_0(S_0)$ over $k_0(x_1,\dots ,x_d)$, we get a uniquely defined embedding
  $$
  k(S)\to k\; ,
  $$
which can be extended to an isomorphism
  $$
  e_P:\overline {k(S)}\stackrel{\sim }{\lra }k\; .
  $$
As the square (\ref{ulitka}) is commutative, 
  $$
  e_P(f)=f(P)
  $$ 
for each $f$ in $k_0[S_0]$.
\end{pf}

\begin{remark}
{\rm Each isomorphism $e_P$ is non-canonical, as it depends on the choice of the transcendental basis $B$ containing the quantities $b_1,\dots ,b_d$.}
\end{remark}

Let 
  $$
  \eta _0=\Spec (k_0(S_0))
  $$ 
be the generic point of the scheme $S_0$, 
  $$
  \eta =\Spec (k(S))
  $$ 
the generic point of the scheme $S$, and let 
  $$
  \bar \eta =\Spec (\overline {k(S)})
  $$ 
be the geometric generic point of $S$. The non-canonical isomorphisms 
  $$
  e_P:\overline {k(S)}\stackrel{\sim }{\lra }k
  $$
over $k_0$, appearing in Lemma \ref{iso}, induce the corresponding non-canonical isomorphisms of schemes
  $$
  \varkappa _P:\Spec (k)\stackrel{\sim }{\lra }\bar \eta 
  $$
over $k_0$.

Let now $\bcX$ be a scheme over $S$ having a model $\bcX _0$ over $k_0$. In other words, one has a cartesian square
  $$
  \diagram
  \bcX \ar[dd]_-{} \ar[rr]^-{f} & & S \ar[dd]^-{} \\ \\
  \bcX _0 \ar[rr]^-{f_0} & & S_0
  \enddiagram
  $$
in the category of schemes over $k_0$. Let $\bcX _{0,\eta _0}$ be the generic fibre of the morphism $f_0$, let $\bcX _{\eta }$ and $\bcX _{\bar \eta }$ be, respectively, the generic and geometric generic fibre of the morphism $f$.

Pulling back the scheme-theoretic isomorphism $\varkappa _P$ to the fibres of the family $f$, we obtain the cartesian square
  $$
  \diagram
  \bcX _P \ar[dd]_-{\varkappa _P} \ar[rr]^-{} & &
  \Spec (k) \ar[dd]^-{\varkappa _P} \\ \\
  \bcX _{\bar \eta } \ar[rr]^-{} & & \bar \eta
  \enddiagram
  $$

Notice that since $\varkappa _P$ on the right is an isomorphism of schemes over $\eta _0$, the morphism $\varkappa _P$ on the left is an isomorphism of schemes over $\bcX _{0,\eta _0}$.

If 
  $$
  k_0\subset L\subset k
  $$ 
is a field subextension of $k/k_0$, the morphsim $\bcX \to \bcX _0$ factorizes through the scheme 
  $$
  {\bcX _0}_L=\bcX _0\times _{\Spec (k_0)}\Spec (L)\; .
  $$
Composing the embedding of the fibre $\bcX _P$ in to the total scheme $\bcX $ with the morphism $\bcX \to {\bcX _0}_L$ we can consider $\bcX _P$ as a scheme over ${\bcX _0}_L$.

Now, let $P$ and $P'$ be two $k$-rational points in the $c$-open subset $U$ in $S$, let 
  $$
  \sigma _{PP'}=e_{P'}\circ e_P^{-1}
  $$ 
be the automorphism of the field $k$ induced by the non-canoniocal isomorphisms from Lemma \ref{iso}, and let 
  $$
  \varkappa _{PP'}=\varkappa _{P'}^{-1}\circ \varkappa _P:\bcX _P\lra \bcX _{P'}
  $$ 
be the induced isomorphism of the fibres as schemes over $\Spec (k^{\sigma _{PP'}})$. 

In these terms, 
  $$
  (\bcX _P)_{\sigma _{PP'}}=\bcX _{P'}\; ,
  $$
the isomorphism 
  $$
  w_{\sigma _{PP'}}:\bcX _{P'}\stackrel{\sim }{\to }\bcX _P
  $$ 
is over the scheme 
  $$
  \bcX _0\times _{\Spec (k_0)}\Spec (k^{\sigma _{PP'}})\; ,
  $$
and 
  $$
  w_{\sigma _{PP'}}=\varkappa _{P'P}\; .
  $$

To see that we just need to use Lemma \ref{iso} and pull-back the scheme-theoretic isomorphisms between points on $S$ to isomorphisms between the corresponding fibres of the morphism $f$.

The assumption that $S$ is affine is not essential. One can cover $S$ by open affine subschemes, construct the system of isomorphisms $\varkappa $ in each affine chart and then construct ``transition isomorphisms" between very general fibres in a smooth family over an arbitrary integral base $S$ of finite type over $k$.

And all the same applies to morphisms of schemes over $S$, of course. Indeed let
   $$
   \xymatrix{
   \bcY \ar[rd]_-{} \ar[rr]^-{\iota } & & \bcX \ar[ld]^-{} \\
   & S & }
   $$
be a morphism of schemes over $S$, with a model 
   $$
   \xymatrix{
   \bcY_0 \ar[rd]_-{} \ar[rr]^-{\iota _0} & & \bcX _0\ar[ld]^-{} \\
   & S_0 & }
   $$
over $k_0$. 

Then, for any $k$-point $P$ in $U$, the obvious diagram
  \begin{equation}
  \label{ulitka3}
  \diagram
  \bcY _P \ar[dd]_-{\varkappa _P} \ar[rr]^-{\iota _P} & &
  \bcX _P \ar[dd]^-{\varkappa _P} \\ \\
  \bcY _{\bar \eta } \ar[rr]^-{\iota _{\bar \eta }}
  & & \bcX _{\bar \eta }
  \enddiagram
  \end{equation}
commutes, and the isomorphisms $\varkappa _{PP'}$ commute with the morphisms $\iota _P$ and $\iota _{P'}$, for any two closed $k$-points $P$ and $P'$ in $U$. 

The following lemma is a consequence of the Suslin-Voevodsky's representability Theorem \ref{SV-theorem}.

\begin{lemma}
\label{ratkappa}
The scheme-theoretic isomorphisms $\varkappa _P$ preserve the algebraic and rational equivalence of algebraic cycles.
\end{lemma}

\begin{pf}
Apply Lemma \ref{fieldtwist} in case when
  $$
  k'=\overline {k(S)}\; ,
  $$
  $$
  \sigma =e_P:\overline {k(S)}\stackrel{\sim }{\lra }k
  $$
and the inverse
  $$
  e_P^{-1}:k\stackrel{\sim }{\lra }\overline {k(S)}\; .
  $$
\end{pf}


Now let us see how these generalities apply in a more specific setting. 

Let again $k$ be an un countable algebraically closed field, let $S$ be an integral scheme of finite type over $k$, and let 
   $$
   \bcX \to S
   $$
be a smooth projective family of surfaces over $S$. Let also
   $$
   \xymatrix{
   \bcC \ar[rd]_-{} \ar[rr]^-{\iota } & & \bcX \ar[ld]^-{} \\
   & S & }
   $$
be a morphism over $S$, such that $\bcC $ is a smooth projective family of curves over $S$ with a section over $S$. In this situation, we have the associated Jacobian fibration
  $$
  \bcA =\Jac _{\bcC /S} \to S
  $$
of the family $\bcC /S$, such that for each point $s\in S$ we have the Jacobian $\bcA _s$ of the curve $\bcC _s$ over the residue field $\kappa (s)$ of the scheme $S$ at $s$. 

In particular, for each $k$-rational point $P$ on $S$ we have the closed immersion
  $$
  \iota _P:\bcC _P\to \bcX _P
  $$
of fibres at $P$ over $k$, and the corresponding Gysin homomorphism
  $$
  {\iota _P}_*:\bcA _P(k)\to A_0(\bcX _P)\; .
  $$
Let then $\bcA _{P,0}$ and $\bcA _{P,1}$ be the abelian subvarieties in $\bcA _P$ over $k$, introduced in Sections \ref{A0} and \ref{A1} respectively. 

Similarly, we have the obvious closed immersion
  $$
  \iota _{\bar \eta }:\bcC _{\bar \eta }\to \bcX _{\bar \eta }
  $$
of geometric generic fibres over $\overline {k(S)}$, and the corresponding Gysin homomorphism
  $$
  {\iota _{\bar \eta }}_*:\bcA _{\bar \eta }(\bar \eta )\to A_0(\bcX _{\bar \eta })\; .
  $$
Let $\bcA _{{\bar \eta },0}$ and $\bcA _{{\bar \eta },1}$ be the corresponding abelian subvarieties in $\bcA _P$ over the field $\overline {k(S)}$.

The commutative diagram (\ref{ulitka3}) gives us the commutative diagram
  \begin{equation}
  \label{nalim}
  \diagram
  \bcA _P \ar[dd]_-{\varkappa _P}
  \ar[rr]^-{{\iota _P}_*} & &
  A_0(\bcX _P)\ar[dd]^-{{\varkappa _P}_*} \\ \\
  \bcA _{\bar \eta }
  \ar[rr]^-{{\iota _{\bar \eta }}_*}
  & & A_0(\bcX _{\bar \eta })
  \enddiagram
  \end{equation}
in which the morphisms $\varkappa _P$ are isomorphisms of schemes over $k_0$, and ${\varkappa _P}_*$ are push-forwards on the Chow groups of degree zero $0$-cycles, guaranteed by Lemma \ref{ratkappa}.

Furthermore, let
   $$
   \Alb _{\bcX /S}
   $$
 be the Albanese scheme of the relative surface $\bcX /S$, i.e. the abelian scheme dual to the connected component $\Pic ^0_{\bcX /S}$ of $0$ of the Picard scheme $\Pic ^0_{\bcX /S}$ of the relative surface $\bcX $ over $S$. 
 
 As we assume that the family $\bcC $ has a section over $S$, the same section is also a section for the family $\bcX $ over $S$. Then we can also consider the universal Albanese morphism 
  $$
  \bcX \to \Alb _{\bcX /S}
  $$
over $S$. 

Assume, moreover, that the base scheme $S$ is regular. Then the universal Albanese morphism induces the Albanese homomorphism
  $$
  \alb :A_0(\bcX /S)\to \Alb _{\bcX /S}(S)
  $$
from the group of relative degree zero $0$-cycles modulo rational equivalence, in the sense of Section 20.1 in \cite{Fulton}, to the group of sections of the Albanese scheme over $S$. 

Also one can define the relative Gysin homomorphism
  $$
  \iota :\bcA (S)\to A_0(\bcX /S)
  $$
from the group of sections of the abelian scheme $\bcA $ over $S$ to the Chow group of relative $0$-cycles, whose composition with the Albanese homomorphism is induced buy the regular homomorphism
  $$
  \bcA \to \Alb _{\bcX /S}
  $$
of abelian schemes over $S$. 

Let then 
  $$
  \bcA _1\to S
  $$
be the connected component containing $0$ of the kernel of this homomorphism, considered as an abelian scheme over $S$. Then $\bcA _1$ is an abelian subscheme in the abelian scheme $\bcA $ over $S$. Then
  $$
  \bcA _{P,1}=(\bcA _1)_P
  $$
and
  $$
  \bcA _{\bar \eta ,1}=(\bcA _1)_{\bar \eta }
  $$
are the fibres of the morphism $\bcA _1\to S$ at $P$ and $\bar \eta $ respectively.

\begin{proposition}
\label{soglasovannost}
For any closed point $P$ in the $c$-open set $U$,
  $$
  \varkappa _P(A_{P,1})=A_{\bar \eta ,1}
  $$
and
 $$
 \varkappa _P(A_{P,0})=A_{\bar \eta ,0}
 $$
\end{proposition}

\begin{pf}
Since the abelian varieties $\bcA _{P,1}$ and $\bcA _{\bar \eta ,1}$ are the fibres of the family $\bcA _1$ over $S$, the first claim is actually true for any closed point $P$ on $S$, not only on $U$, and can be easily deduced using specialization isomorphisms on \'etale cohomology groups.

For the second claim, let $\Xi _P$ be the countable subset in $A_P$ and $\Xi _{\bar \eta }$ the countable subset in $A_{\bar \eta }$, such that we have the presentations of the Gysin kernels
  $$
  \Gys _{0,P}(k)=\cup _{x\in \Xi _P}(x+A_{P,0}(k))
  \qand
  \Gys _{0,\bar {\eta }}(\bar \eta )=
  \cup _{x\in \Xi _{\bar \eta }}(x+A_{\bar \eta ,0}(\bar \eta ))
  $$
in $A_P(k)$ and $A_{\bar \eta }(\bar \eta )$ respectively, as provided by Theorem \ref{gysstructure}. Then
  $$
  \varkappa _P(\Gys _{0,P}(k))=
  \varkappa _P(\cup _{x\in \Xi _P}(x+A_{P,0}(k)))=
  \cup _{x\in \Xi _P}(\varkappa _P(x)+\varkappa _P(A_{P,0}(k)))
  $$
The definition of $\varkappa _P$ and the commutative diagram (\ref{nalim}) give us that
  $$
  \varkappa _P(\Gys _{0,P}(k))=\Gys _{0,\bar \eta }(\bar \eta )\; .
  $$
Therefore,
  $$
  \cup _{x\in \Xi _P}(\varkappa _P(x)+\varkappa _P(A_{P,0}(k)))=
  \cup _{x\in \Xi _{\bar \eta }}(x+A_{\bar \eta ,0}(\bar \eta ))
  $$
inside the abelain group $A_{\bar \eta }(\bar \eta )$. 

Now, since $\varkappa _P$ are regular morphisms of schemes over $k_0$, we obtain that $\varkappa _P(A_{P,0})$ is a Zariski closed subset in $A_{\bar \eta }$. And as $\varkappa _P(A_{P,0}(k))$ is a subgroup in the group $A_{\bar \eta }(\bar \eta )$, the scheme-theoretic image $\varkappa _P(A_{P,0})$ is an abelian subvariety in the abelian variety $A_{\bar \eta }$. Lemma \ref{syroezhka2} and Lemma \ref{syroezhka3} finish the proof.
\end{pf}

\section{Gysin kernels in Lefschetz pencils}
\label{gys}

\subsection{$l$-adic monodromy and the Picard-Lefschetz formula}
\label{adicmonodromy}

Let $k$ be an algebraically closed field, and let $X$ be a smooth projective surface over $k$. Choose and fix an appropriate closed immersion of $X$ in to some projective space $\PR $, and a Lefschetz pencil on $X$ over $k$, see Proposition 1.5 on page 177 in \cite{FreitagKiehl} or Expos\'e XVII in \cite{SGA7-2}. Resolving the indeterminacy locus by blowing up a finite number of points on $X$, we shall assume, without loss of generality, that our pencil 
  $$
  f:X\to \PR ^1
  $$ 
is a regular morphism having a section 
  $$
  \PR ^1\to X
  $$
over $k$, see Theorem 31.3 on page 185 in \cite{MilneLEC} or Expos\'e XVII in \cite{SGA7-2}. 

Let $U$ be the Zariski open subset in $\PR ^1$, such that the fibre $C_t$ of the pencil $f$ at $t$ is smooth, for every closed point $t$ in $U$, and let
  $$
  \iota _t:C_t\to X
  $$
be the corresponding closed immersion of the fibre $C_t$ in to $X$.  

Let 
  $$
  K=k(\PR ^1)
  $$ 
be the function field of $\PR ^1$, and let
  $$
  \overline K
  $$
be the algebraic closure of the field $K$. Let also
  $$
  \eta =\Spec (K)
  $$
be the generic point of the projective line, and let
  $$
  \bar \eta =\Spec (\overline K)
  $$
be the geometric generic point of the line $\PR ^1$. 

The generic fibre $C_{\eta }$ of the Lefschetz pencil is defined over $\eta $. Accordingly, we obtain the closed immersion
  $$
  \iota _{\eta }:C_{\eta }\to X\times _k\eta 
  $$
over $\eta $. Extending scalars, we also have the closed immersion
  $$
  \iota _{\bar \eta }:C_{\bar \eta }\to X\times _k{\bar \eta }
  $$
over the geometric generic point $\bar \eta $.

Let
  $$
  A=\Jac _{C_{\eta }/\eta }
  $$
be the Jacobian of the curve $C_{\eta }$ over $\eta $, and let
  $$
  A_{\bar \eta }=A\times _{\eta }\bar \eta 
  $$
be the corresponding extension of scalars. 

Since the field $\overline {k(\PR ^1)}$ is algebraically closed, we also have the abelian subvariety
  $$
  A_{\bar \eta ,1}\subset A_{\bar \eta }\; ,
  $$
a particular case of the abelian variety $A_1$ introduced in Section \ref{A1}. 

Assume that the ground field $k$ is uncountable. Then the closure of the function field $k(\PR ^1)$ is uncountable too, and, therefore, we obtain the abelian subvariety
  $$
  A_{\bar \eta ,0}\subset A_{\bar \eta ,1}\; ,
  $$
a particular case of the abelian variety $A_0$ introduced in Section \ref{A0}, and associated with the closed immersion $\iota _{\bar \eta }$.

Let now 
  $$
  K\subset L\subset \overline K
  $$ 
be the minimal subextension of fields, such that all three abelian varieties $A_{\bar \eta ,0}$, $A_{\bar \eta ,1}$ and $A_{\bar \eta }$, as well as the closed immersions of $A_{\bar \eta ,0}$ in to $A_{\bar \eta ,1}$ and $A_{\bar \eta ,1}$ in to $A_{\bar \eta }$ are all defined over $L$. 

Write
  $$
  D=\PR ^1
  $$
and choose a smooth projective curve $D'$ together with a finite morphism
  $$
  D'\to D
  $$
over $k$, such that
  $$
  L=k(D')\; ,
  $$
and the extension $k(D)\subset k(D')$ is induced by the morphism from $D'$ to $D$. 

Let
  $$
  \eta '=\Spec (k(D'))
  $$
be the generic point of the curve $D'$. Since $L$ is a subfield in $\overline K$, if $\eta '$ is the generic point of $D'$, the geometric generic point $\bar \eta '$ remains the same,
  $$
  \bar \eta '=\bar \eta \; .
  $$

Spreading our abelian varieties over an appropriate Zariski open subset $U'$ in $D'$, we obtain abelian schemes 
  $$
  \alpha :\bcA \to U'\; ,
  $$ 
  $$
  \alpha _1:\bcA _1\to U'
  $$ 
and 
  $$
  \alpha _0:\bcA _0\to U'
  $$
whose generic fibres are the models of our abelian varieties over the minimal field of definition $L$,
  $$
  (\bcA )_{\eta '}=A_L\; ,
  $$
  $$
  (\bcA _1)_{\eta '}=A_{L,1}
  $$
and 
  $$
  (\bcA _0)_{\eta '}=A_{L,0}\; ,
  $$
respectively. Then, of course, 
  $$
  (\bcA )_{\bar \eta }=A_{\bar \eta }\; ,
  $$
  $$
  (\bcA _1)_{\bar \eta }=A_{\bar \eta ,1}
  $$
and 
  $$
  (\bcA _0)_{\bar \eta }=A_{\bar \eta ,0}\; .
  $$
  
Also we have the obvious morphisms
  $$
  \bcA _1\to \bcA 
  $$
and
  $$
  \bcA _0\to \bcA _1
  $$
over $U'$, such that, when passing to the fibres at the geometric generic point $\bar \eta $, we obtain the closed immersions
  $$
  A_{\bar \eta ,1}\to A_{\bar \eta }
  $$
and
  $$
  A_{\bar \eta ,0}\to A_{\bar \eta ,1}
  $$
over $\overline K$.

Notice also that, since $\bcA $ is a spread of $A_{\eta }$ over $U'$ and $A_{\eta }$ is a projective variety over $L$, the morphism $\alpha $ is locally projective and, therefore, proper, see Tag 01WC in \cite{StacksProject}. For the same reason, the morphisms $\alpha _0$ and $\alpha _1$ are proper. Cutting more points from $D'$ we may assume that the morphisms $\alpha $, $\alpha _0$ and $\alpha _1$ are all smooth over $U'$, see Section 3.2 in \cite{Poonen}.

The varieties $A_{\bar \eta }$, $A_{\bar \eta ,1}$ and $A_{\bar \eta ,0}$ are non-canonically isomorphic to corresponding varieties at closed points in a $c$-open subset in $\PR ^1$, as explained in Section \ref{gengeneric}. Namely, choose an appropriate $c$-open subset $U$ in $D=\PR ^1$, such that the fibre $C_t$ is smooth and the point $\bar \eta $ scheme-theoretically is isomorphic to each closed point $t$ in $U$, as in Section \ref{gengeneric}. 

Let
  $$
  A_t=\Jac _{C_t/k}
  $$
be the Jacobian of $C_t$. Since the ground field $k$ is algebraically closed, we have the abelian subvariety 
  $$
  A_{t,1}\subset A_t\; ,
  $$
and $k$ is uncountable, we also have the abelian subvariety
  $$
  A_{t,0}\subset A_{t,1}\; ,
  $$
for each closed point $t$ in $U$. Since the point $\bar \eta $ scheme-theoretically is isomorphic to each closed point $t$ in $U$, these isomorphisms induce scheme-theoretic isomorphisms 
  $$
  A_{\bar \eta ,1}\simeq A_{t,1}
  $$
and 
  $$
  A_{\bar \eta ,0}\simeq A_{t,0}\; ,
  $$
for each closed point $t$ in $U$.

Next, let $\eta '$ be the generic point of $D'$, let $\bar \eta '=\bar \eta $ be the geometric generic point of $D'$, let $\pi _1(U',\bar \eta )$ be the \'etale fundamental group of $D'$ pointed at $\bar \eta $. For any scheme $V$ and non-negative integer $n$ let $(\ZZ /l^n)_V$ be the constant sheaf on $V$ associated to the group $\ZZ /l^n$.

Since the morphisms $\alpha _0$, $\alpha _1$ and $\alpha $ are smooth and proper, the higher direct images
  $$
  R^1{\alpha _0}_*(\ZZ /l^n)_{\bcA _0}\; ,\; \; \;
  R^1{\alpha _1}_*(\ZZ /l^n)_{\bcA _1}
  \qqand R^1\alpha _*(\ZZ /l^n)_{\bcA }
  $$
are locally constant by Theorem 8.9, Ch. I in \cite{FreitagKiehl}. Then the fibres of these sheaves at the geometric generic point $\bar \eta $ are finite continuous $\pi _1(U',\bar \eta )$-modules, see Proposition A I.7 in \cite{FreitagKiehl}. The proper base change (see, for example, Theorem $6.1'$ on page 62 in loc. cit.) gives us the identification of the geometric generic fibres with the corresponding \'etale cohomology groups,
  $$
  (R^1{\alpha _0}_*(\ZZ /l^n)_{\bcA _0})_{\bar \eta }=
  H^1_{\et }({\bcA _0}_{\bar \eta },\ZZ /l^n)\; ,
  $$
  $$
  (R^1{\alpha _1}_*(\ZZ /l^n)_{\bcA _1})_{\bar \eta }=
  H^1_{\et }({\bcA _1}_{\bar \eta },\ZZ /l^n)
  $$
and
  $$
  (R^1\alpha _*(\ZZ /l^n)_{\bcA })_{\bar \eta }=
  H^1_{\et }(\bcA _{\bar \eta },\ZZ /l^n)\; .
  $$
Then we obtain that $\pi _1(U',\bar \eta )$ acts continuously on
  $$
  H^1({\bcA _0}_{\bar \eta },\ZZ /l^n)\; ,
  $$
  $$
  H^1({\bcA _1}_{\bar \eta },\ZZ /l^n)
  $$
and
  $$
  H^1(\bcA _{\bar \eta },\ZZ /l^n)\; .
  $$

Passing to limits on $n$, and then tensoring with $\QQ _l$ we then obtain that $\pi _1(U',\bar \eta )$ acts continuously on the $l$-adic cohomology groups
  $$
  H^1({\bcA _0}_{\bar \eta },\QQ _l)=
  H^1(A_{\bar \eta ,0},\QQ _l)\; ,
  $$
  $$
  H^1({\bcA _1}_{\bar \eta },\QQ _l)=H^1_{\et }(A_{\bar \eta ,1},\QQ _l)
  $$
and
  $$
  H^1(\bcA _{\bar \eta },\QQ _l)=H^1_{\et }(A_{\bar \eta },\QQ _l)\; .
  $$

We also need the (Poincar\'e dual) Gysin homomorphisms
  $$
  g_0:H^1(A_{\bar \eta ,0},\QQ _l)\to H^1(A_{\bar \eta },\QQ _l)
  $$
and 
  $$
  g_1:H^1(A_{\bar \eta ,1},\QQ _l)\to
  H^1(A_{\bar \eta },\QQ _l)\; ,
  $$
see Section \ref{A1}. The action of $\pi _1(U',\bar \eta )$ naturally commutes with both $g_0$ and $g_1$.

The standard theory of \'etale monodromy, see \cite{WeilConjI} or \cite{FreitagKiehl}, gives us that the \'etale fundamental group $\pi _1(U,\bar \eta )$ acts continuously on the . The latter group acts in the cohomology group  
  $$
  H^1(C_{\bar \eta },\QQ _l)\; ,
  $$
and this action is tame. More precisely, let
  $$
  \pi ^{\tame }_1(U,\bar \eta )
  $$
be the tame \'etale fundamental group of $U$, which can be considered as the quotient group of $\pi _1(U,\bar \eta )$. Then the action of the latter on cohomology passes through wild ramification, after shrinking $U$, which means that the above action factorizes bringing the action of $\pi ^{\tame }_1(U,\bar \eta )$ on $H^1(C_{\bar \eta },\QQ _l)$, see Theorem 7.1 on pp 247 - 248 in \cite{FreitagKiehl}. 

Since the morphism from $D'$ to $D$ is finite, the fundamental group $\pi _1(U',\bar \eta )$ is a subgroup in the group $\pi _1(U',\bar \eta )$, and the same for tame fundamental groups. Applying the same arguments to the obvious morphism 
  $$
  f_{U'}:X_{U'}\to U'
  $$
we obtain the continuous action of the fundamental group $\pi _1(U',\bar \eta )$ on the same cohomology group $H^1(C_{\bar \eta },\QQ _l)$.

Next, we are going to use the Picard-Lefschetz formulae, and that's why we need to recall the local-to-global generation of the tame fundamental group, for the convenience of the reader. We will follow \cite{FreitagKiehl}. 

Let
  $$
  S=D\smallsetminus U
  $$
be the set of bad points on $D=\PR ^1$, where we have singular fibres and wild ramification, and for each closed point 
  $$
  s\in S\; ,
  $$
let
  $$
  \delta _s\in H^1(C_{\bar \eta },\QQ _l)
  $$
be the unique up-to conjugation vanishing cycle corresponding to the point $s$ in the standard sense (see Theorem 7.1 on page 247 in \cite{FreitagKiehl}). Let also
  $$
  E\subset H^1(Y_{\bar \eta },\QQ _l)
  $$
be the $\QQ _l$-vector subspace generated by all the elements $\delta _s$, $s\in S$. That is, $E$ is the space of vanishing cycles in $H^1(C_{\bar \eta },\QQ _l)$. Then $E$ is nothing but the kernel of the Gysin homomorphism
  $$
  H^1(C_{\bar \eta },\QQ _l)\to H^3(X_{\bar \eta },\QQ _l(1))\; ,
  $$
induce by the closed immersion of the geometric generic fibre $C_{\bar \eta }$ in to the surface 
  $$
  X_{\bar \eta }=X\times \bar \eta \; ,
  $$
see Section 4.3 in \cite{WeilConjII} or consult \cite{FreitagKiehl}. Moreover, the space $E$ is a $\pi _1^{\tame }(U,\bar \eta )$-submodule in $H^1(C_{\bar \eta },\QQ _l)$, see \cite{WeilConjI}, \cite{WeilConjII} and \cite{FreitagKiehl}.

Let, 
  $$
  \mu _n=\{ \alpha \in \overline K\, |\, \alpha ^n=1\}
  $$
be the multiplicative group of $n$-th root of unity in the algebraic closure of the function field of the projective line, let  
  $$
  p=\left\{
  \begin{array}{ll}
  p\; , & \mbox{if $\cha (k)=p>0$} \\
  1 & \mbox{if $\cha (k)=0$}
  \end{array}
  \right.
  $$
be the exponential characteristic of the ground field $k$, and let  
  $$
  \hat \ZZ ^{(p)}(1)=\lim _{(n,p)=1}\mu _n
  $$
be the limit of the inverse system of all groups $\mu _n$ with $n$ coprime to $p$ and the obvious transition homomorphisms 
  $$
  \mu _{n'}\to \mu _n\; ,
  $$
whenever $n|n'$. Then
  $$
  \hat \ZZ ^{(p)}(1)\simeq \prod _{l\neq p}\ZZ _l(1)
  $$
is the product over all primes $l$ not equal to $p$, and we obtain the canonical homomorphism 
  $$
  \hat \ZZ ^{(p)}(1)\to \ZZ _l(1)\; ,
  $$
sending
  $$
  u\to \bar u\; ,
  $$
for each such prime number $l$. 

On the other hand, for each $s$ we have a homomorphism 
  $$
  \gamma _s:\hat \ZZ ^{(p)}(1)\to \pi _1^{\tame }(U,\bar \eta )\; ,
  $$
and let
  $$
  I_s=\im (\gamma _s)
  $$
be the image of it. Then $\gamma _s$ and $I_s$ are defined up to conjugation inside the tame group $\pi _1^{\tame }(U,\bar \eta )$, see the details on pp 246 - 247 in \cite{FreitagKiehl}, and compare it with another view of the same thing given on pp 186 - 187 in \cite{MilneLEC}. The group $I_s$ can be considered as a local tame fundamental group at $s$.

All the subgroups $I_s$, for all $s\in S$, generate the tame fundamental group topologically. In other words, the smallest closed subgroup containing all the subgroups $I_s$ is the whole tame fundamental group $\pi _1^{\tame }(U,\bar \eta )$. This is sufficient to control submodules in the $\pi _1^{\tame }(U,\bar \eta )$-module $E$ by looking at the action by elements in $I_s$ and using the Picard-Lefschetz formula. 

Recall that the Picard-Lefschetz formula tells us that, for any
  $$
  x\in H^1(C_{\bar \eta },\QQ _l)\; ,
  $$
and for any
  $$
  u\in I_s\; ,
  $$
one has


   \begin{equation}
   \label{PicardLefschetz}
   \gamma _s(u)x=
   x\pm \bar u\langle x,\delta _s\rangle \delta _s\; ,
   \end{equation}
   
\medskip

\noindent see Theorem 7.1 on pages 247 - 248, or page 190 in \cite{MilneLEC}.

\subsection{Deligne's irreducibility and the first dichotomy (Theorem A)}
\label{irreducibility}

In this section we will prove Theorem A sated in Introduction. We need some more lemmas.

\medskip

\begin{lemma}
\label{lem:purely-insep-uh}
Let $f:X\to Y$ be a finite surjective morphism between two integral schemes over $k$, such that the function field $k(X)$ is purely inseparable over the function field $k(Y)$. Then there exists a nonempty open subscheme $U$ in $Y$, such that $X_U=f^{-1}(U)\to U$ is a universal homeomorphism of schemes.
\end{lemma}

\begin{pf}
Let $\eta_Y$ and $\eta_X$ be the generic points of $Y$ and $X$ respectively.  The base change to $\eta _Y$ of the diagram
  $$
  \diagram
  X \ar[ddrr]_-{} \ar[rr]^-{\Delta } & & X\times _YX \ar[dd]^-{} \\ \\
  & & Y
  \enddiagram
  $$
gives us the diagram
  $$
  \diagram
  \eta _X \ar[ddrr]_-{} \ar[rr]^-{\Delta _{\eta }} & & 
  \eta _X\times _{\eta _Y}\eta _X \ar[dd]^-{} \\ \\
  & & Y
  \enddiagram
  $$
The morphism $\Delta _{\eta }$ is surjective by Tag 01S4 in \cite{StacksProject}. By Tag 07RR in \cite{StacksProject}, there exists a nonempty open $U$ of $Y$ such that the morphism 
  $$
  \Delta _U:X_U\to X_U\times_UX_U
  $$ 
is surjective. By Tag 01S4 again, the morphism $X_U\to U$ is radicial. Since it is also finite and surjective, our result follows.
\end{pf}

\begin{lemma}
\label{chernika}
Let $X$ and $Y$ be two smooth, connected schemes of dimension $1$, separated and of finite type over $k$, and let $f:X\to Y$ be a finite flat surjective morphism over $k$. Let $\bar \eta$ be the geometric generic point of the scheme $Y$. Then there exists a nonempty open subset $U\subset Y$, such that for all nonempty open subsets $W\subset U$ the \'etale fundamental group $\pi _1(f^{-1}(W),\bar \eta )$ is a subgroup of finite index in the \'etale fundamental group $\pi _1(W,\bar \eta )$, and the same for tame fundamental groups.
\end{lemma}

\begin{proof}
By the correspondence between algebraic curves and fields of transcendence degree $1$ over $k$, we can factor $f$ as a composition
  $$
  X\stackrel{u}{\lra }X'\stackrel{f'}{\lra }Y
  $$
in a way, such that the field $k(X)$ is purely inseparable over the field $k(X')$, and $k(X')$ is separable over $k(Y)$. Then we have the following diagram of cartesian squares:
  $$
  \xymatrix{
  X_{V'} \ar[rr]^-{u_{V'}} \ar[dd]_-{} & & V' \ar[dd]_-{} & & \\ \\
  X_V \ar[rr]^-{u_V} \ar[dd]_-{} & & X'_V \ar[rr]^-{f'_V} \ar[dd]_-{} & &
  V \ar[dd]_-{} \\ \\
  X \ar[rr]^-{u} & &
  X' \ar[rr]^-{f'} & &
  Y
  }
  $$
where 
  $$
  V\subset Y\qqand V' \subseteq X'_V
  $$ 
are nonempty open subschemes, such that the morphism $f'_V$ is finite \'etale, and $u_{V'}$ is a universal homeomorphism by Lemma \ref{lem:purely-insep-uh}.

Shrinking $V$ to $V\smallsetminus f'_V(X'_V\smallsetminus V')$, we may assume that $u_V$ is a universal homeomorphism too. 

Then for any open $W\subseteq V$, we have a cartesian diagram
  $$
  \xymatrix{
  X_W \ar[rr]^-{u_W} \ar[dd]_-{} & & X'_W \ar[rr]^-{f'_W} \ar[dd]_-{} & &
  W \ar[dd]_-{} \\ \\
  X_V \ar[rr]^-{u_V} & & X'_V \ar[rr]^-{f'_V} & & V
  }
  $$
where $u_W$ is a universal homeomorphism, and $f_W'$ is a finite \'etale.

Then, by Tag 0BQN in \cite{StacksProject}, we obtain an isomorphism
  $$
  \pi _1(X_W,\bar \eta )\simeq \pi _1(X'_W,\bar \eta )\; .
  $$

Since $f'_W$ is finite \'etale, the group $\pi _1(X'_W,\bar \eta )$ is a subgroup of finite index in $\pi _1(W,\bar \eta )$ by Section 1.6, Expos\'e 10 in \cite{SGA1}.

The claim for tame fundamental groups follows from the proven case.
\end{proof}

\begin{lemma}
\label{porechki}
Let $G$ be a group, and let $H$ be a subgroup of finite index $n$ in $G$. Then, for every element $g\in G$, 
  $$
   g^{n!}\in H\; .
  $$
\end{lemma}

\begin{proof}
Consider the action of the group $G$ on the set of left cosets $G/H$ by left multiplication:
  $$
  G\times G/H\to G/H\; , 
  $$
  $$
  g\cdot (xH)=(gx)H\; .
  $$
Since 
  $$
  |G/H|=[G:H]=n\; ,
  $$
this action induces a homomorphism
  $$
  \varphi :G\lra S_n\; ,
  $$
where $S_n$ is the group of permutations of the set $G/H$.

Hence, for each $g\in G$,
  $$
  g^{n!}\in \ker(\varphi )\subset H\; .
  $$
\end{proof}

\medskip

Now we are ready to prove our first dichotomy theorem (Theorem A in Introduction).

\medskip

\begin{theorem}
\label{firstdichotomy}
Under the assumptions above, either $A_{\bar \eta ,0}=0$ or $A_{\bar \eta ,0}=A_{\bar \eta ,1}$.
\end{theorem}

\medskip

\begin{pf}
Since our pencil
  $$
  f:X\to D
  $$
admits a section, we have a point on $C_{\eta }$ rational over $\eta $. Use this point to embed the curve $C_{\eta }$ in to its Jacobian
  $$
  A_{\eta }=\Jac _{C_{\eta }/\eta }
  $$
over $\eta $. Extending scalars, the close immersion
  $$
  C_{\bar \eta }\to A_{\bar \eta }
  $$
over $\bar \eta $ induces an isomorphism on $l$-adic cohomology groups,
  \begin{equation}
  \label{shitiki}
  H^1(A_{\bar \eta },\QQ _l)\stackrel{\sim }{\lra }
  H^1(C_{\bar \eta },\QQ _l)\; ,
  \end{equation}
which is compatible with the action of the tame fundamental group $\pi _1^{\tame }(U,\bar \eta )$. This allows us to identify both groups and consider the space of vanishing cycles $E$ as a $\pi _1^{\tame }(U,\bar \eta )$-submodule in the $\pi _1^{\tame }(U,\bar \eta )$-module $H^1(A_{\bar \eta },\QQ _l)$. 

Define a homomorphism $\zeta $ as in Section \ref{A1},
  \begin{equation}
  \label{boroviki}	
  \diagram
  H^1(A_{\bar \eta ,1},\QQ _l) \ar[dd]_-{\sim } \ar[rr]^-{\zeta _1} & & 
  H^1(A_{\bar \eta },\QQ _l) \\ \\
  H_1(A_{\bar \eta ,1},\QQ _l(-1)) \ar[rr]^-{{i_1}_*} & & 
  H_1(A_{\bar \eta },\QQ _l(-1)) \ar[uu]^-{\sim } 
  \enddiagram
  \end{equation}

Now, let
  $$
  X_{\bar \eta }=X\times _{\Spec (k)}{\bar \eta }
  $$
be the scalar extension of the surface $X$, and consider the Gysin homomorphism on cohomology
  $$
  \iota _*:H^1(C_{\bar \eta },\QQ _l)\to H^3(X_{\bar \eta },\QQ _l(1))\; .
  $$
By Poincar\'e duality, we have the commutative square 
  $$
  \diagram
  H^1(C_{\bar \eta },\QQ _l) \ar[dd]_-{\sim } \ar[rr]^-{\iota _*} & & 
  H^3(X_{\bar \eta },\QQ _l(1)) \ar[dd]^-{\sim } \\ \\
  H_1(C_{\bar \eta },\QQ _l(-1)) \ar[rr]^-{\iota _*} & & 
  H_1(X_{\bar \eta },\QQ _l(-1))  
  \enddiagram
  $$
in which the vertical arrows are isomorphisms of $\QQ _l$-vector spaces. Applying Lemma \ref{H1Albanese}, we also obtain the commutative square
  $$
  \diagram
  H_1(C_{\bar \eta },\QQ _l(-1)) \ar[dd]_-{\sim } \ar[rr]^-{\iota _*} & & 
  H_1(X_{\bar \eta },\QQ _l(-1)) \ar[dd]^-{\sim } \\ \\
  H_1(A_{\bar \eta },\QQ _l(-1)) \ar[rr]^-{\alpha _*} & & 
  H_1(B_{\bar \eta },\QQ _l(-1))  
  \enddiagram
  $$
where
  $$
  A_{\bar \eta }=\Jac _{C_{\bar \eta }/{\bar \eta }}=
  \Alb _{C_{\bar \eta }/{\bar \eta }}\; ,
  $$
  $$
  B_{\bar \eta }=\Alb _{X_{\bar \eta }/{\bar \eta }}=
  \Alb _{X/k}\times _{\Spec (k)}{\bar \eta }\; ,
  $$
and $\alpha $ is the homomorphism induced by closed immersion $\iota $ on Albanese varieties.

Combining the last two commutative square we obtain the commutative square 
  \begin{equation}
  \label{mohoviki}
  \diagram
  H^1(C_{\bar \eta },\QQ _l) \ar[dd]_-{\sim } \ar[rr]^-{\iota _*} & & 
  H^3(X_{\bar \eta },\QQ _l(1)) \ar[dd]^-{\sim } \\ \\
  H_1(A_{\bar \eta },\QQ _l(-1)) \ar[rr]^-{\alpha _*} & & 
  H_1(B_{\bar \eta },\QQ _l(-1))  
  \enddiagram
  \end{equation}

By Lemma \ref{dualities}, the sequence
  $$
  0\to H_1(A_{\bar \eta ,1},\QQ _l(-1))\stackrel{{i_1}_*}{\lra }
  H_1(A_{\bar \eta },\QQ _l(-1))
  \stackrel{\alpha _*}{\lra }H_1(B_{\bar \eta },\QQ _l(-1))\to 0
  $$
is exact. Therefore, combining (\ref{shitiki}), (\ref{boroviki}) and (\ref{mohoviki}), we obtain the resulting short exact sequence
  $$
  0\to H^1(A_{\bar \eta ,1},\QQ _l)\stackrel{\zeta _1}{\lra }
  H^1(C_{\bar \eta },\QQ _l)\stackrel{{\iota }_*}{\lra } 
  H^3(X_{\bar \eta },\QQ _l(1))\to 0\; ,
  $$
where $\zeta _1$ is the composition of the old $\zeta _1$ and the isomorphism (\ref{shitiki}). 

Thus, the image of the injective homomorphism $\zeta '_1$ coincides with the kernel of the Gyusin homomorphism $\iota _*$ on cohomology groups. And since the $\QQ _l$-vector space of vanishing cycles $E$ coincides with the kernel of this Gysing homomoirphism, we can identify $E$ with $H^1(A_{\bar \eta ,1},\QQ _l)$. Thus,
  $$
  E=H^1(A_{\bar \eta ,1},\QQ _l)\; .
  $$

Moreover, if
  $$
  X_{\eta }=X\times _{\Spec (k)}\eta 
  $$
is the scalar extension, spreading out the closed immersion 
  $$
  \iota _{\eta }:C_{\eta }\to X_{\eta }
  $$
over a sufficiently small Zariski open subset $U$ in $\PR ^1$, one can verify that the injective homomorphism $\zeta _1$ is $\pi _1(U',\bar \eta )$-equivariant, where $U'$ is the pullback of $U$ under the morphism from $D'$ to $D=\PR ^1$. Therefore, $E$ and $H^1(A_{\bar \eta ,1},\QQ _l)$ can be identified as $\pi _1(U',\bar \eta )$-modules, and hence as $\pi ^{\tame }_1(U',\bar \eta )$-modules under the tame action. 

Now, by Lemma \ref{chernika}, the tame fundamental group $\pi ^{\tame }_1(U',\bar \eta )$ is a subgroup of finite index $n$ in the tame fundamental group $\pi ^{\tame }_1(U,\bar \eta )$. 

Since $C=C_{\eta }$ is a curve,
  $$
  \langle \delta _s,\delta _s\rangle =0\; ,
  $$
by the Picard-Lefschetz formula, we have that
  $$
  \gamma _s(u)^{n!}\cdot x=
  x\pm n! \, \bar u \langle x,\delta _s\rangle \delta _s\; .
  $$

By Lemma \ref{porechki}, for any 
  $$
  x\in E_0\; ,
  $$
we have 
  $$
  \gamma _s(u)^{n!}\cdot x\in E_0\; ,
  $$
hence
  $$
  \bar u\langle x,\delta _s\rangle \delta _s=
  \pm \frac{1}{n!}(\gamma _s(u)^{n!}x-x)\in E_0\; ,
  $$
and so 
  $$
  \gamma _s(u)\cdot x \in E_0\; .
  $$
  
By Theorem 7.1 (3) in Chapter 3 of \cite{FreitagKiehl}, the action of $\pi _1(U,\bar \eta )$ on $H^1(C_{\bar \eta },\QQ _l)$ is tame. Since the set 
  $$
  \{\sigma \in \pi _1^{\tame }(U,\bar \eta )\; | \; \sigma E_0 \subset E_0\}
  $$
is closed in $\pi _1^{\tame }(U,\bar \eta )$, it follows that $E_0$ is a $\pi _1^{\tame }(U,\bar \eta )$-submodule of $E$. 

By Deligne, the action of $\pi _1^{\tame }(U,\bar \eta )$ on $E$ is absolutely irreducible, see Corollary 7.4 on page 249 in \cite{FreitagKiehl} or Corollary 5.5 on page 291 in the original Deligne's paper \cite{WeilConjI}. It follows that either $E_0=0$ or $E_0=E$.

In case 
  $$
  E_0=0\; ,
  $$
we have that 
  $$
  H^1(A_{\bar \eta ,0},\QQ _l)=0\; ,
  $$
whence 
  $$
  A_{\bar \eta ,0}=0\; .
  $$
  
If 
  $$
  E_0=E\; ,
  $$
then 
  $$
  H^1(A_{\bar \eta ,0},\QQ _l)=H^1(A_{\bar \eta ,1},\QQ _l)\; ,
  $$
and then
  $$
  A_{\bar \eta ,0}=A_{\bar \eta ,1}
  $$
for dimensional reason. 

The theorem is proven.
\end{pf}

\subsection{Representability of $0$-cycles and second dichotomy (Theorem B)}

Let $k$ be any algebraically closed field of zero or positive characteristic, such as $\bar \QQ $ or $\bar \FF _p$, for a prime $p$, for example. Let also
  $$
  k\subset \ud 
  $$
be a field extension, such that $\ud $ is an uncountable universal domain over $k$, in the sense of Weil, see \cite{Weil}. If the characteristic of $k$ is $0$, then $\ud $ can be the field of complex numbers $\CC $. If the characteristic is $p>0$, then $\ud $ can be taken to be the algebraic closure $\overline {\FF _p((t))}$ of the field of Laurent series with coefficients in $\FF _p$. The latter may be also considered as the field of generalised Puiseux series over $\FF _p$, see \cite{Kedlaya}.

Let $X$ be a smooth projective variety over $k$, and let 
  $$
  X_{\ud }=X\times _k\ud 
  $$
be its scalar extension to the universal domain. For each $d\in \NN $, consider the homomorphism 
  $$
  \cl _d:\Sym ^d(X(\ud ))\times \Sym ^d(X(\ud ))\to A_0(X_{\ud })
  $$
introduced in Section \ref{grpcompletion0ycles}. For any two $0$-cycles 
  $$
  A,B\in \Sym ^d(X(\ud ))\; ,
  $$
the homomorphism $\cl _d$ sends 
  $$
  (A,B)\mapsto [A-B]\; ,
  $$
where $[A-B]$ is the cycle class of $A-B$ in $A_0(X_{\ud })$. Following Mumford, we say that the group $A_0(X_{\ud })$ is finite-dimensional, if there exists $d\in \NN $, such that the homomorphism $\cl _d$ is surjective, see item (2) on page 196 in \cite{Mumford}. This is one of the equivalent definitions of finite-dimensionality of the group $A_0(X_{\ud })$, see items (1) and item (3) on the same page of Mumford's paper. The complete list of equivalent reformulations of finite-dimensionality of $A_0(X_{\ud })$ is given in Proposition 1.6 on page 253 of the illuminating Jannsen's paper \cite{JannsenMotivicSheaves}. 

Let
  $$
  \alb _{X_{\ud }}:A_0(X_{\ud })\to \Alb _{X_{\ud }/\ud }(\ud )
  $$
be the Albanese homomorphism of the variety $X_{\ud }$, and let
  $$
  T(X_{\ud })=\ker (\alb _{X_{\ud }})
  $$
be the kernel of it. By Jannsen's result, the group $A_0(X_{\ud })$ is finite-dimensional in Mumford's sense if and only if the Albanese kernel is trivial,
  $$
  T(X_{\ud })=0\; ,
  $$
see Items (v) and (v') in Proposition 1.6 in \cite{JannsenMotivicSheaves}. 

A closely related concept is Bloch's weak representability of Chow groups, see Definition 1.1 in \cite{BlochAnExample} and Definition 3.3 on page 63 in \cite{BlochMurre}. In particular, if $n$ is the dimension of $X$, the group $A_0(X)$ is said to be weakly representable, if there exists a smooth projective curve $\Gamma $ over $k$, a cycle class 
  $$
  \zeta \in CH^n(\Gamma \times X)\; ,
  $$
and an algebraic subgroup 
  $$
  G\subset J_{\Gamma }
  $$ 
in the Jacobian $J_{\Gamma }$ of the curve $\Gamma $, such that the induced homomorphism
  $$
  \zeta _*:J_{\Gamma }(\ud )=A_0(\Gamma_{\ud })\to A_0(X_{\ud })
  $$
is surjective with the kernel $G(\ud )$. If $X$ is a smooth projective surface over $k$, weak representability of the group $A_0(X_{\ud })$ is equivalent to finite-dimensionality of this group, see Lemma 3.9 on page 66 in \cite{BlochMurre}. 

\medskip

Recall also that, given an algebraic variety $X$ over $\ud $, a $c$-closed subset in $X$ is the union of a countably many Zariski closed subsets in $X$, and a $c$-open subset in $X$ is the complement to a $c$-closed subset in $X$. 

\begin{lemma}
\label{c-moving}
Let $C$ be a smooth projective curve over $\ud $, and let $U$ be a $c$-open subset in $C$. Then, any closed point $P$ on $C$ is rationally equivalent to a $0$-cycle on $C$ supported on $U$. 
\end{lemma}

\begin{pf}
Choose a closed point $P_0\in U$, and consider the linear system 
  $$
  \Sigma =|P+nP_0|
  $$ 
on $C$, for $n\in \NN $. By Riemann-Roch, 
  $$
  l(P+nP_0)=l(K_C-P-nP_0)+n+2-g\; ,
  $$
where $K_C$ is a canonical divisor on $C$. If the degree $1+n$ is greater than the degree of the canonical divisor, $2g - 2$, we obtain that
  $$
  l(K_C-P-nP_0)=0\; ,
  $$
and then 
  $$
  l(P+nP_0)=n+2-g\; .
  $$
It means that
  $$
  \dim (\Sigma )=n+1-g\; .
  $$

Thus, if $n>g-1$, the dimension of $\Sigma $ is positive. Since the set of $0$-cycles in $\Sigma $ supported on the countable complement to $U$ in $C$ is countable, and the projective space $\Sigma $ itself is uncountable, there exists a $0$-cycle $A$ in $\Sigma $ supported on $U$. Then $P$ is rationally equivalent to the $0$-cycle 
  $$
  A-nP_0
  $$ 
supported on $U$. 
\end{pf}

\begin{corollary}
\label{c-moving-2}
Let $X$ be an algebraic variety over $\ud $, and let $U$ be a $c$-open subset in $X$. Then any $0$-cycle on $X$ is rationally equivalent to a $0$-cycle supported on $U$. 
\end{corollary}

\begin{pf}
For any point $P$ on $X$, choose an irreducible curve $C$ passing through $P$ on $X$. Let $\tilde C$ be the resolution of singularities (normalization) of the curve $C$ and let $\tilde P$ be a point on $\tilde C$ sent to $P$ under the morphism from $\tilde C$ to $C$. By Lemma \ref{c-moving}, the point $\tilde P$ is rationally equivalent to a $0$-cycle $\tilde A$ on $\tilde C$ supported on the inverse image of $U$ under the morphism from $\tilde C$ to $X$. Then $P$ is rationally equivalent to the push-forward $A$ of the $0$-cycle $\tilde A$ from $\tilde C$ to $X$, and this push-forward is supported on $U$. As this is true for any closed point $P$ on $X$, any $0$-cycle on $X$ is rationally equivalent to a $0$-cycle supported on $U$. 
\end{pf}

Let $X$ be a smooth projective variety over $\ud $, and let
  $$
  f:X\to T
  $$
be a regular morphism on to some irreducible projective curve $T$ over $\ud $. For each closed point $t$ in $T$, let $C_t$ be the fibre of the morphism $f$ at $t$. 

\begin{proposition}
\label{surj1}
For any non-empty $c$-open subset $U$ in $T$, the sum 
  $$
  \oplus _{P\in U}CH_0(C_P)\to CH_0(X)
  $$
of Gysin homomorphisms over all closed points of $U$ is surjective.
\end{proposition}

\begin{pf}
The set 
  $$
  Z=T\smallsetminus U
  $$
is countable in $T$, and therefore the set
  $$
  X_U=f^{-1}(U)
  $$
is $c$-open in $X$ as the complement to the union of the fibres at points of $Z$. By Corollary \ref{c-moving-2}, any point on $X$ is rationally equivalent to a $0$-cycle supported on $X_U$. Since $X_U$ is fibred over $U$, we obtain that any $0$-cycle on $X$ is rationally equivalent to a sum of $0$-cycles supported on the fibres of the morphism $f|_U$ from $X_U$ to $U$.
\end{pf}

\begin{corollary}
\label{surj2}
For any non-empty $c$-open subset $U$ in $T$, the homomorphism
  $$
  \oplus _{P\in U}A_0(C_P)\to A_0(X)
  $$
is surjective.
\end{corollary}

\begin{pf}
Take any $0$-cycle $A$ of degree $0$ on $X$. By Proposition \ref{surj1}, it is in the image of the sum of Gysin homomorphism from $CH_0(C_P)$ to $CH_0(X)$, for $P$ running over $U$. In other words, there exist close points $P_1,\ldots ,P_n$ in $U$, and $0$-cycles $A_i\in C_{P_i}$ for each $i\in \{ 1,\ldots ,n\} $, such that $A$ is rationally equivalent to $A_1+\ldots +A_n$ on $X$. Let $d_i$ be the degree of $A_i$. Since the degree of $A$ is $0$, it follows that the sum $d_1+\ldots +d_n$ is $0$ too. Choose a point $Q_i$ in each fibre $C_{P_i}$. Then each $0$-cycle $A_i-d_iQ_i$ is of degree $0$ and supported on $C_{P_i}$. Therefore, it represents an element in $A_0(C_{P_i})$. Clearly, $A$ is rationally equivalent to the sum $\sum _i(A_i-d_iQ_i)$. 
\end{pf}

\medskip

Let now $X$ be a smooth projective surface over $k$. As above, choose a suitable closed immersion of $X$ in to some projective space $\PR $, and fix a Lefschetz pencil on $X$ over $k$. As above, resolving the indeterminacy locus, we may assume that our pencil 
  $$
  f:X\to \PR ^1
  $$ 
is a regular morphism having a section over $k$. We will be using the results and terminology from Sections \ref{adicmonodromy} and \ref{irreducibility}. 

In particular, let 
  $$
  \eta =\Spec (k(t))
  $$
be the generic point of $\PR ^1$, let 
  $$
  C=C_{\eta }
  $$
be the generic fibre of the Lefschetz pencil, and let
  $$
  \iota :C\to X_{\eta }
  $$
be the closed immersion of the generic fibre in to the surface
  $$
  X_{\eta }=X\times \eta 
  $$
over $\eta $. 

Let also
  $$
  \iota _{\bar {\eta }}:C_{\bar {\eta }}\to X_{\bar {\eta }}
  $$
 be the closed immersion of the geometric generic fibre, i.e. the scalar extension of the immersion $\iota $ from the field $k(t)$ to the algebraic closure $\overline {k(t)}$. 
 
 Fix an embedding 
   $$
   k(t)\subset \ud 
   $$
 over $k$. Then
   $$
   \overline {k(t)}\subset \ud \; ,
   $$
 and we also have the closed immersion
   $$
   \iota _{\ud }:C_{\ud }\to X_{\ud }
   $$
 over the uncountable universal domain $\ud $. 
 
 As above, let
   $$
   A=\Jac _{C/\eta }
   $$
be the Jacobian of the curve $C$, and consider the Gysin kernel 
   $$
   \Gys _0(\overline {k(t)})=
   \ker ({\iota _{\bar \eta }}_*:A_{\bar \eta }\to A_0(X_{\bar \eta }))
   $$
over the field $\overline {k(t)}$, and the Gysin kernel
   $$
   \Gys _0(\ud )=
   \ker ({\iota _{\ud }}_*:A_{\ud }\to A_0(X_{\ud }))
   $$
over the universal domain, where $A_{\bar \eta }$ and $A_{\ud }$ are the scalar extensions.

Clearly,
  $$
   \Gys _0(\overline {k(t)})\subset \Gys _0(\ud )\; ,
   $$
and our second dichotomy theorem (Theorem B in Introduction) is this:

\begin{theorem}
\label{seconddichotomy}

In terms above, if $A_0(X_{\ud })$ is not finite-dimensional, then the Gysin kernel $\Gys _0(\ud )$ is countable. Assuming the Gysin homomorphism
  $$
  H^1(C_{\bar \eta },\QQ _l)\to H^3(X_{\bar \eta },\QQ _l(1))
  $$
is not injective, the converse is also true: if $\Gys _0(\ud )$ is countable, then $A_0(X_{\ud })$ is not weakly representable. 
\end{theorem}

\begin{pf}
Suppose the Chow group $A_0(X_{\ud })$ is not finite-dimensional, and the Gysin kernel $\Gys _0(\ud )$ is not countable. Then, by the first dichotomy theorem, Theorem \ref{firstdichotomy}, 
  $$
  A_{\bar \eta ,0}=A_{\bar \eta ,1}\; ,
  $$
where $A_{\bar \eta ,1}$ is the connected component of $0$ in the kernel $G_1(\bar \eta )$ of the surjective homomorphism 
  $$
  \alpha _{\bar \eta }:A_{\bar \eta }\to B_{\bar \eta }\; ,
  $$
where 
  $$
  B_{\bar \eta }=\Alb _{X_{\bar \eta }/\bar \eta }\; ,
  $$
and $A_{\bar \eta ,0}$ is the structural abelian variety for the Gysin kernel $\Gys _0(\bar \eta )$, induced by the closed immersion $\iota _{\bar \eta }$, see Section \ref{A0}. 
  
Then, as we have shown in Section \ref{gengeneric}, there exists a $c$-open subset
  $$
  U\subset \PR ^1_{\ud }\; ,
  $$
such that, for each closed point 
  $$
  P\in U
  $$
we have the equality
  $$
  A_{P,0}=A_{P,1}\; ,
  $$  
see Proposition \ref{soglasovannost}. Here $A_{P,1}$ is the connected component of $0$ in the kernel $G_1(k)$ of the surjective homomorphism 
  $$
  \alpha _P:A_P\to B\; ,
  $$
where 
  $$
  A_P=\Jac _{C_P/\ud }
  $$
and
  $$
  B=\Alb _{X_{\ud }/\ud }\; ,
  $$
and $A_{P,0}$ is the structural abelian variety for the Gysin kernel $\Gys _0(k)$,  induced by the closed immersion 
  $$
  \iota _P:C_P\to X_{\ud }\; .
  $$
  
Now, each homomorphism $\alpha _P$ induces an isogeny
  $$
  A_P/A_{P,1}\to B\; ,
  $$
see Section \ref{A1}. Therefore, we have a short exact sequence
  $$
  0\to A_{P,1}(\ud )_{\QQ }\lra A_P(\ud )_{\QQ }\lra B(\ud )_{\QQ }\to 0\; .
  $$
Since $A_{P,0}$ is equal to $A_{P,1}$, we obtain a homomorphism $\theta _P$ as making the diagram
  $$
  \diagram
  0 \ar[r]^-{} & A_{P,1}(\ud )_{\QQ } \ar[rr]^-{} & & 
  A_P(\ud )_{\QQ } \ar[ddrr]^-{} \ar[rr]^-{} & & 
  B(\ud )_{\QQ } \ar[dd]^-{\theta _P} \ar[r]^-{} & 0 \\ \\
  & & & & & A_0(X_{\ud })_{\QQ } & 
  \enddiagram
  $$
commute. 

The curves $C_P$ are smooth fibres of the pencil $f:C\to \PR ^1$. It follows that the homomorphisms $\theta _P$ do not depend on $P$,
  $$
  \theta =\theta _P\; .
  $$

Applying Corollary \ref{surj2} to the surface $X_{\ud }$, we see that the composition 
  $$
  A_0(X_{\ud })_{\QQ }\stackrel{\alb }{\lra }
  \Alb _{X_{\ud }/\ud }(\ud )_{\QQ }=B(\ud )_{\QQ }
  \stackrel{\theta }{\lra }
  A_0(X_{\ud })_{\QQ }
  $$
is the identity. It follows that
  $$
  T(X_{\ud })_{\QQ }=0\; .
  $$
Since the Albanese kernel is torsion free by Ro\u \i tman's theorem, see \cite{Roitman} and \cite{MilneRoitmanThm}, we obtain that
  $$
  T(X_{\ud })=0\; ,
  $$
as required. 

Now, suppose that the cohomological Gysin homomorphism
  $$
  H^1(C_{\bar \eta },\QQ _l)\to H^3(X_{\bar \eta },\QQ _l(1))
  $$
is not injective. It follows that the homomorphism
  $$
  \alpha _{\bar \eta }:A=\Jac _{C/\bar \eta }\to 
  B=\Alb _{X_{\bar \eta }/\bar \eta }
  $$
is not an isogeny over $\overline {k(t)}$. Then
  $$
  \alpha _{\ud }:A_{\ud }\to B_{\ud }
  $$
is not an isogeny over the universal domain $\ud $. Since $C$ is a hyperplane section of $X$, the homomorphism $\alpha $ is surjective, and so is the homomorphism $\alpha _{\ud }$. It follows that its kernel is not finite, whence
  $$
  A_{\bar \eta ,1}\neq 0\; .
  $$

Therefore, if $A_0(X_{\ud })$ is finite-dimensional, then $T(X_{\ud })=0$ and 
  $$
  A_{\bar \eta ,0}=A_{\bar \eta ,1}\neq 0\; .
  $$
Hence $\Gys _0(\ud )$ contains the set $A_{\bar \eta ,0}(\ud )$, which is uncounable. 
\end{pf}

\subsection{Transcendental elements and third dichotomy (Theorem C)}

The aim of this last section is to prove Theorem C in the introduction. The key ingredient here is again the Suslin-Voevodsky's representability theorem and the proven Theorems A and B. Let us start with some simple lemmas. 

\begin{lemma}
\label{simplelemma1}
Let $k$ be a field, and let $S$ be a set of elements algebraically independent over $k$. Let $S_1$ and $S_2$ be two disjoint subsets in $S$. Then
  $$
  \overline {k(S_1)}\cap \overline {k(S_2)}=\overline k
  $$
inside $\overline {k(S)}$.
\end{lemma}

\begin{pf}
Without loss of generality, we may assume that the sets $S_1$ and $S_2$ are finite. 

Let 
  $$
  \alpha \in \overline {k(S)}
  $$ 
be an element algebraic over both $k(S_1)$ and $k(S_2)$. Consider the diagram of field extensions 

  $$
   
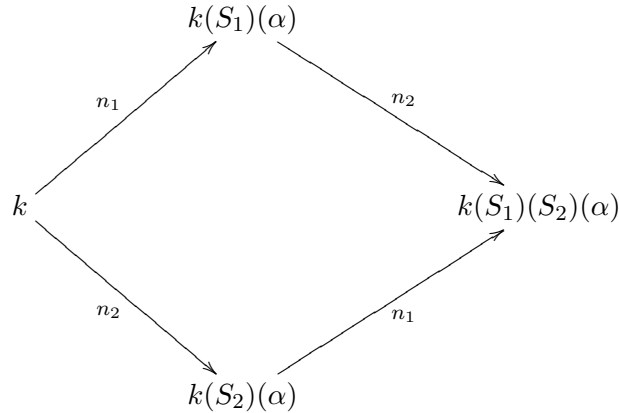
\xymatrix{
   & & k(S_1)(\alpha )\ar[rrdd]^-{n_2} & & \\ \\
   k\ar[rruu]^-{n_1}\ar[rrdd]_-{n_2} & & & & k(S_1)(S_2)(\alpha )  \\ \\
   & & k(S_2)(\alpha )\ar[rruu]_-{n_1} & & }
   $$

\bigskip

\noindent where $n_i$ is the number of elements in $S_i$ for $i=1,2$.

By Theorem 1.1 in Chapter 8 of Lang's book {\red ???}, there exists a transcendence basis 
  $$
  B\subset S_2
  $$ 
for the extension 
  $$
  k(S_1)(\alpha )\subset k(S_1)(\alpha )(S_2)\; .
  $$
Since the transcendence degree of this extension is $n_2$, it follows that 
  $$
  B=S_2\; .
  $$
That is, $S_2$ is algebraically independent over $k(S_1)(\alpha )$.

But the element $\alpha $ is algebraic over the field $k(S_2)$. Therefore, there exists a nonzero polynomial 
  $$
  f(S_2,t)\in k[S_2][t]\; ,
  $$ 
such that 
  $$
  f(S_2,\alpha )=0\; .
  $$
  
Rewriting this as 
  $$
  f(t)(S_2)\in k[t][S_2]\; ,
  $$
we see that 
  $$
  f(\alpha )(S_2)=0
  $$ 
in the field 
  $$
  k(\alpha )(S_2)\subset k(S_1)(\alpha)(S_2)\; .
  $$

By the independence of $S_2$ over $k(S_1)(\alpha )$, this means that all the coefficients of the polynomial 
  $$
  f(\alpha )(S_2)=\sum _Ig_I(\alpha )S_2^I
  $$ 
must vanish in the field $k(\alpha )$. That is, 
  $$
  g_I(\alpha )=0
  $$ 
for all indices $I$. 

Since $f\neq 0$, at least one $g_I\neq 0$, this implies that $\alpha $ is algebraic over the field $k$. Hence, 
  $$
  \alpha \in \overline k\; .
  $$
\end{pf}

\begin{lemma}
\label{simplelemma2}
Let $L/k$ be a field extension, and let $X$ be a scheme over $k$. For any subfields 
  $$
  k_1,k_2\subset L
  $$ 
containing $k$, one has
  $$
  X(k_1)\cap X(k_2)=X(k_1\cap k_2)\; ,
  $$
where the intersection is taken inside $X(L)$.
\end{lemma}

\begin{pf}
A $k_i$-rational point of $X$, for $i=1,2$, is a morphism of schemes 
  $$
  \Spec (k_i)\to X
  $$
over $k$. By functoriality, the sets $X(k_1)$ and $X(k_2)$ are naturally subsets in $X(L)$, and we are considering their intersection inside $X(L)$.

Suppose 
  $$
  x\in X(k_1)\cap X(k_2)\; .
  $$
Then $x$ corresponds to morphisms 
  $$
  \Spec (k_1) \to X\qqand \Spec (k_2)\to X
  $$ 
that agree when composed with the inclusions of $k_1$ and $k_2$ in to the field $L$. Then they must both factor through a common subfield over which the morphism is defined. This common subfield is actually $k_1\cap k_2$.

Thus $x$ arises from a morphism 
  $$
  \Spec (k_1\cap k_2)\to X\; .
  $$
That is,
  $$
  x\in X(k_1 \cap k_2)\; .
  $$

Conversely, any morphism $\Spec (k_1\cap k_2) \to X$ induces morphisms 
  $$
  \Spec (k_1)\to X\qqand \Spec (k_2)\to X\; ,
  $$
which means that 
  $$
  X(k_1\cap k_2)\subset X(k_1)\cap X(k_2)\; .
  $$
Hence,
  $$
  X(k_1)\cap X(k_2)=X(k_1 \cap k_2)\; .
  $$
\end{pf}

\begin{corollary}
\label{simplecorollary}
Let $k$ be a field, let $S$ be a set of elements algebraically independent over $k$, and let $S_1$ and $S_2$ be two disjoint subsets of $S$. Then
  $$
  X(\overline {k(S_1)})\cap X(\overline {k(S_2)})=X(\overline k)\; .
  $$
\end{corollary}

\begin{pf}
By Lemma \ref{simplelemma2}, we have that
  $$
  X(\overline {k(S_1)})\cap X(\overline {k(S_2)})=
  X(\overline {k(S_1)}\cap \overline {k(S_2)})\; .
  $$
By Lemma \ref{simplelemma1}, since $S_1$ and $S_2$ are disjoint and algebraically independent over $k$, their algebraic closures intersect trivially over $\overline k$. In other words,
  $$
  \overline {k(S_1)}\cap \overline {k(S_2)}=\overline k\; .
  $$
Therefore,
  $$
  X(\overline {k(S_1)})\cap X(\overline {k(S_2)})=X(\overline k)\; .
  $$
\end{pf}

\medskip

Let now again $X$ be a smooth projective surface over an algebraically closed field $k$, and let again 
  $$
  f:X\to \PR ^1
  $$ 
be our Lefschetz pencil over $k$. 

Let $k_0$ be the minimal field of definition of the pencil $f$. In other words, there exists a morphism 
  $$
  f_0:X_0\to \PR ^1_{k_0}
  $$
over $k_0$, and the pullback square
  $$
  \diagram
  X \ar[dd]_-{} \ar[rr]^-{} & & \PR ^1_k \ar[dd]^-{} \\ \\
  X_0 \ar[rr]^-{} & & \PR ^1_{k_0}
  \enddiagram
  $$
in the category of schemes over $k_0$, such that $X_0$ is smooth projective, and $k_0$ is the composite of the minimal field of definition of $X$ and the graph of the pencil $f$, in the sense of Weil, see \cite{Weil}. 

Then 
  $$
  k_1=k_0(t)
  $$
is the minimal field of definition of the generic fibre 
  $$
  C=C_{\eta }
  $$
of the pencil $f$.

As above, we fix an embedding
  $$
  k(t)\subset \ud \; .
  $$
Let then 
  $$
  \bar k_1\subset \ud 
  $$ 
be the algebraic closure of the field $k_1$ inside $\ud $. Let also
  $$
  C_{\bar k_1}=C_0\times _{k_0(t)}\bar k_1
  $$
and 
  $$
  X_{\bar k_1}=X_0\times _{k_0(t)}\bar k_1\; .
  $$
Let 
  $$
  \Gys _0(\bar k_1)
  $$ 
be the Gysin kernel induced by the obvious closed immersion 
  $$
  \iota _{\bar k_1}:C_{\bar k_1}\to X_{\bar k_1}\; ,
  $$
and let 
  $$
  \Gys _0(\ud )
  $$ 
be the Gysin kernel induced by the closed immersion 
  $$
  \iota _{\ud }:C_{\ud }\to X_{\ud }\; .
  $$
Naturally,
  $$
  \Gys _0(\bar k_1)\subset \Gys _0(\ud )\; .
  $$

Our last theorem (Theorem C in Introduction) sheds light on arithmetic of Gysin kernel $G_0(\ud )$, provided it is countable.

\begin{theorem}
If the Gysin kernel $G_0(\ud )$ is countable, then 
  $$
  \Gys _0(\ud )=\Gys _0(\bar k_1)\; .
  $$
\end{theorem}

\begin{pf}
Let $S$ be a transcendence basis of the field extension
  $$
  \bar k_1\subset \ud \; .
  $$
As the universal domain is uncountable, so is the set $S$. 
  
Fix a finite subset 
  $$
  S_0\subset S\; .
  $$
Then we can find uncountably many subsets 
  $$
  S_i\subset S\; ,\; \; \; i\in I\; ,
  $$
such that $S_i$ has the same number of elements as the fixed set $S_0$, for all indices $i\in I$, and the sets $S_i$ are pairwise disjoint,
  $$
  S_i\cap S_j=\emptyset 
  $$
for all $i,j\in I$. 

And we can also assume that 
  $$
  0\in I\; ,
  $$
and then $S_0$ is among these subsets in $S$. 

As usual, let
  $$
  A=\Jac _{C/\eta }
  $$
be the Jacobian of the generic fibre $C$ over $\eta $. By Corollary \ref{simplecorollary}, 
  $$
  A(\overline {k_1(S_i)})\cap A(\overline {k_1(S_j)})=A(\bar k_1)\; .
  $$
Then 
  $$
  \Gys _0(\overline {k_1(S_i)})\cap \Gys _0(\overline {k_1(S_j)})
  \subset A(\bar k_1)\cap \Gys _0(\ud )\; ,
  $$
and since
  $$
  A(\bar k_1)\cap \Gys _0(\ud )\subset \Gys _0(\bar k_1)\; ,
  $$
we obtain
  $$
  \Gys _0(\overline {k_1(S_i)})\cap \Gys _0(\overline {k_1(S_j)})
  \subset \Gys _0(\bar k_1)\; .
  $$
But 
  $$
  \Gys _0(\bar k_1)\subset \Gys _0(\overline {k_1(S_i)})
  $$
and
  $$
  \Gys _0(\bar k_1)\subset \Gys _0(\overline {k_1(S_j)})\; .
  $$
Therefore,
  $$
  \Gys _0(\overline {k_1(S_i)})\cap \Gys _0(\overline {k_1(S_j)})=
  \Gys _0(\bar k_1)\; ,
  $$
for any two indices $i$ and $j$ in $I$. 

This gives a disjoint union of sets
  \begin{equation}
  \label{muhomor}
  \coprod _{i\in I}\Gys _0(\overline {k_1(S_i)})\smallsetminus 
  \Gys _0(\bar k_1)
  \end{equation}
inside the Gysin kernel $\Gys _0(\ud )$.

The point here is that the field $\overline {k_1(S_i)}$ is isomorphic to the field $\overline {k_1(S_j)}$ over $\bar k_1$, for each $i$ and $j$. 

Applying Corollary \ref{fieldtwist}, we obtain the commutative diagram
  $$
  \diagram
  0 \ar[r]^-{} & \Gys _0(\bar k_1) \ar[dd]_-{\id } \ar[rr]^-{} & & 
  \Gys _0(\overline {k_1(S_i)}) \ar[dd]_-{\sim } \ar[rr]^-{} & & 
  A(\overline {k_1(S_i)}) \ar[dd]_-{\sim } \\ \\
  0 \ar[r]^-{} & \Gys _0(\bar k_1) \ar[rr]^-{} & & 
  \Gys _0(\overline {k_1(S_j)}) \ar[rr]^-{}  & &
  A(\overline {k_1(S_i)})
  \enddiagram
  $$  
which shows that the sets
  $$
  \Gys _0(\overline {k_1(S_i)})\smallsetminus 
  \Gys _0(\bar k_1)
  $$
are of the same cardinality, for all $i$ in $I$. 

Thus, (\ref{muhomor}) is a disjoint union of uncountably many sets of tyhe same cardinality on $\Gys _0(\ud )$. Therefore, since the latter set is uncountable by assumption, it follows that
  $$
  \Gys _0(\overline {k_1(S_i)})=\Gys _0(\bar k_1)\; ,
  $$
for each $\in I$. 

In particular,
  $$
  \Gys _0(\overline {k_1(S_0)})=\Gys _0(\bar k_1)\; .
  $$

Since $S_0$ was an arbitrary finite subset in $S$, it follows that
  $$
  \Gys _0(\ud )=\Gys _0(\bar k_1)\; .
  $$
\end{pf}

\bigskip

\bigskip

\begin{small}

\end{small}

\bigskip

\bigskip


{\sc Department of Mathematical Sciences, University of Liverpool, Peach Street, Liverpool L69 7ZL, England, UK}


\medskip


{\it E-mail}: {\tt vladimir.guletskii@liverpool.ac.uk}


\bigskip


{\sc School of Mathematical Sciences, Capital Normal University, Beijing 100048, China}


\medskip


{\it E-mail}: {\tt bozhang9511@gmail.com}



\begin{thebibliography}{9999999}



\bibitem[SGA1]{SGA1}
Rev$\hat {\rm e}$tements \'etales et groupe fondamental. S\'eminaire de G\'eom\'etrie Alg\'ebrique du Bois Marie 1960 - 61 (SGA 1) dirig\'e par A. Grothendieck. Augment\'e de deux expos\'es de M. Raynaud. Lecture Notes in Mathematics. Vol. 224 (1971)




\bibitem[SGA7-2]{SGA7-2}
Groupes de monodromie en g\'eom\'etrie alg\'ebrique. S\'eminaire de G\'eom\'etrie Alg\'ebrique du Bois Marie 1967-69 (SGA 7-2) par P. Deligne, N. Katz. Lecture Notes in Mathematics 340 (1973)

\bibitem[Stacks]{StacksProject}
The Stacks Project. {\small {{\url {https://stacks.math.columbia.edu}}}}

\bibitem{Anderson}
O. Anderson. Relative algebraic cycles. PhD thesis 2018

\bibitem{Kalyan&Vovan}
K. Banerjee, V. Guletski\u \i . \'Etale monodromy and rational equivalence for $1$-cycles on cubic hypersurfaces in $\PR ^5$. Matematicheskii Sbornik. Volume 211, Issue 2 (2020) 161 - 200

\bibitem{Debarre}
O. Debarre. Higher-dimensional algebraic geometry. Springer 2001

\bibitem{FreitagKiehl}
E. Freitag, R. Kiehl. Etale Cohomology and the Weil Conjecture. Ergebnisse der Mathematik und ihrer Grenzgebiete. 3 Folge. A Series of Modern Surveys in Mathematics. Volume 13. Springer-Verlag 1988

\bibitem{Fulton}
W. Fulton. Intersection theory. Ergebnisse der Mathematik und ihrer Grenzgebiete. 3 Folge. Band 2. Springer-Verlag 1984



\bibitem{BCinitial}
S. Bloch. $K_2$ of artinian $\QQ $-algebras, with application to algebraic cycles. Communications in Algebra. Volume 3, Issue 5 (1975) 405 - 428

\bibitem{BlochAnExample}
S. Bloch. An example in the theory of algebraic cycles. Lecture Notes in Mathematics 551 (1976) 1 - 29

\bibitem{BlochMurre}
S. Bloch, J. Murre. On the Chow group of certain types of Fano threefolds. Compositio Mathematica. Volume 39, Issue 1 (1979) 47 - 105

\bibitem{BlochLectures}
S. Bloch. Lectures on Algebraic Cycles. Duke University Math. Series IV. Durham, NC. Duke University 1980



\bibitem{Debarre}
O. Debarre. Higher-Dimensional Algebraic Geometry. Springer 2001

\bibitem{WeilConjI}
P. Deligne. La conjecture de Weil I. Publications Math\'ematiques de l'IH\'ES 43 (1974) 273 - 307

\bibitem{WeilConjII}
P. Deligne. La conjecture de Weil II. Publications Math\'ematiques de l'IH\'ES 52 (1980) 137 - 252



\bibitem{FreitagKiehl}
E. Freitag, R. Kiehl. Etale Cohomology and the Weil Conjecture. Ergebnisse der Mathematik und ihrer Grenzgebiete. 3 Folge. Band 13. Springer-Verlag 1988


\bibitem{Hartshorne}
R. Hartshorne. Algebraic Geometry. Springer 1977

\bibitem{HartshorneAmple}
R. Hartshorne. Ample subvarieties of algebraic varieties. Lecture Notes in Mathematics 156. Springer 1970


\bibitem{JannsenMotivicSheaves}
U. Jannsen. Motivic Sheaves and Filtratins on Chow Groups. In "Motives", Proc. Symposia in Pure Math. Vol. 55, Part 1 (1994) 245 - 302

\bibitem{Jouanolou}
J.-P. Jouanolou. Th\'eor\`emes de Bertini et Applications. Birkh\"auser 1983

\bibitem{Kedlaya}
K. Kedlaya. The algebraic closure of the power series field in positive characteristic. Proc. Amer. Math. Soc. 129 (12) (2001) 3461 - 3470

\bibitem{Kollar}
J. Koll\'ar. Rational curves on algebraic varieties. Springer 1996

\bibitem{Lang}
S. Lang. Introduction to algebraic geometry. Interscience Publishers 1958

\bibitem{LangAV}
S. Lang. Abelian varieties. Interscience Publishers 1959


\bibitem{Laumon}
G. Laumon. Homologie \'etale. Ast\'ersique. Tome 36 - 37 (1976) 163 - 188



\bibitem{MilneAV}
J. Milne. Abelian varieties. Available at {\small {{\url {https://www.jmilne.org/math/CourseNotes/AV.pdf}}}}

\bibitem{MilneLEC}
J. Milne. Lectures on \'Etale Cohomology. Available at {\small {{\url {https://www.jmilne.org/math/CourseNotes/LEC.pdf}}}}

\bibitem{MilneRoitmanThm}
J. Milne. Zero cycles on algebraic varieties in nonzero characteristic: Rojtman's theorem. Compositio Mathematica. Vol. 47, No. 3 (1982) 271 - 287


\bibitem{Mumford}
D. Mumford. Rational equivalence of $0$-cycles on surfaces. J. Math. Kyoto Univ. 9 (1969) 195 - 204



\bibitem{PaninAEMHT}
I. Panin (after I. Panin and A. Smirnov). Riemann-Roch theorems for oriented cohomology. In Proceedings of the NATO Advanced Study Institute on Axiomatic, Enriched and Motivic Homotopy Theory. Edited by J. Greenlees (2004) Springer-Science + Business Media, B.V.

\bibitem{PaninHHA}
I. Panin. Oriented cohomology theories of algebraic varieties II. In Homology Homotopy Applications, Vol. 11, No. 1 (2009) 349 - 405

\bibitem{Polishchuk}
A. Polishchuk. Abelian Varieties, Theta Functions and the Fourier Transform.  Cambridge University Press 2003

\bibitem{Poonen}
B. Poonen. Rational points on varieties. Graduate Texts in Mathematics 2017



\bibitem{RoitmanGamma}
A. Ro\u \i tman. On $\Gamma $-equivalence of zero-dimensional cycles. Math. USSR-Sb. Vol. 15 No. 4 (1971) 555 - 567

\bibitem{Roitman}
A. Roitman. The torsion of the group of $0$-cycles modulo rational equivalence. Annals of Mathematics. Second Series 111 (3) (1980) 553 - 569


\bibitem{RydhThesis}
D. Rydh. Families of cycles and the Chow scheme. PhD thesis. May 2008


\bibitem{SchoemannWerner}
C. Schoemann, S. Werner. The kernel of the Gysin homomorphism for positive characteristic. Available at {\small {{\url {https://arxiv.org/abs/2411.11417v3}}}}

\bibitem{SuslinVoevodsky}
A. Suslin, V. Voevodsky. Singular homology of abstract algebraic varieties. Inventiones mathematicae. Vol. 123, Issue 1 (1996) 61 - 94




\bibitem{VoisinBook}
C. Voisin. Hodge Theory and Complex Algebraic Geometry, Volume I $\& $ II. Cambridge studies in advanced mathematics 76 $\& $ 77 (2002)

\bibitem{VoisinSymplInv}
C. Voisin. Symplectic involutions of $K3$-surfaces act trivially on $CH_0$. Documenta Mathematica 17 (2012) 851 - 860


\bibitem{Weibel}
C. Weibel. The K-book: An Introduction to Algebraic K-theory. Graduate Studies in Mathematics 145. AMS 2013

\bibitem{Weil}
A. Weil. Foundations of Algebraic Geometry. AMS 1946

\end{thebibliography}
\end{document}